\documentclass[a4paper, 11pt, final]{article}

\usepackage{enumerate,graphicx,amsfonts,psfrag}
\usepackage{epsfig}

\newtheorem{proposition}{Proposition}[section]
\newtheorem{theorem}[proposition]{Theorem}
\newtheorem{lemma}[proposition]{Lemma}

\newtheorem{definition}[proposition]{Definition}
\newtheorem{remark}[proposition]{Remark}

\newenvironment{proof}{\smallskip\noindent\emph{\textbf{Proof.}}\hspace{1pt}}%
{\hspace{-5pt}{\nobreak\quad\nobreak\hfill\nobreak$\square$\vspace{8pt}%
\par}\smallskip\goodbreak}

\newenvironment{proofof}[1]{\smallskip\noindent\emph{\textbf{Proof of #1.}}%
\hspace{1pt}}{\hspace{-5pt}{\nobreak\quad\nobreak\hfill\nobreak%
$\square$\vspace{8pt}\par}\smallskip\goodbreak}

\newcommand{\Section}[1]{\section{#1}\setcounter{equation}{0}}

\renewcommand{\div}{{\mathrm{div}}}
\newcommand{\Div}{{\mathrm{Div}}}
\renewcommand{\L}[1]{\mathbf{L^#1}}
\newcommand{\Lloc}[1]{\mathbf{L^{#1}_{loc}}}
\newcommand{\C}[1]{\mathbf{C^{#1}}}
\newcommand{\Cc}[1]{\mathbf{C_c^{#1}}}

\newcommand{\W}[2]{\mathbf{W^{#1,#2}}}
\newcommand{\modulo}[1]{{\left|#1\right|}}
\newcommand{\norma}[1]{{\left\|#1\right\|}}
\newcommand{\reali}{{\mathbb{R}}}
\newcommand{\tv}{\mathrm{TV}}
\newcommand{\BV}{\mathbf{BV}}
\newcommand{\Lip}{\mathinner\mathbf{Lip}}
\renewcommand{\epsilon}{\varepsilon}
\renewcommand{\phi}{\varphi}

\newcommand{\rpic}{\overline\reali_+}
\newcommand{\rpis}{\reali_+}

\newcommand{\iiiint}{\int\!\!\int\!\!\int\!\!\int}

\newcommand{\pt}{\partial}
\newcommand{\dt}{\partial_t}

\newcommand{\ds}{\partial_s}

\newcommand{\pd}{{\eta}}
\newcommand{\gd}{{\lambda}}

\title{Stability and Total Variation Estimates\\ on General Scalar
  Balance Laws}

\author{Rinaldo M.~Colombo, Magali Mercier\thanks{Permanent address:
    Universit\'e de Lyon, Universit\'e Lyon 1; 43, bd.~du 11 novembre
    1918; 69622 Villeurbanne Cedex}\/ and Massimiliano
  D.~Rosini\thanks{Supported by INdAM.}\\ Department of Mathematics,
  Brescia University\\ Via Branze 38, 25133 Brescia\\ Italy}

\begin{document}

\maketitle

\begin{abstract}
  \noindent Consider the general scalar balance law $\partial_t u +
  \Div f(t, x,u) = F(t,x,u)$ in several space dimensions. The aim of
  this note is to estimate the dependence of its solutions from the
  flow $f$ and from the source $F$. To this aim, a bound on the total
  variation in the space variables of the solution is obtained. This
  result is then applied to obtain well posedness and stability
  estimates for a balance law with a non local source.

  \smallskip

  \noindent\textit{2000~Mathematics Subject Classification:} 35L65.

  \smallskip


\end{abstract}

\Section{Introduction}
\label{sec:intro}

The Cauchy problem for a scalar balance law in $N$ space dimension
\begin{equation}
  \label{eq:probv}
  \left\{\begin{array}{l@{\qquad}rcl}
      \partial_t u + \Div  f(t,x,u) = F(t,x,u)
      & (t,x) & \in & \rpis\times\reali^N
      \\
      u(0,x)=u_o(x) & x & \in & \reali^N
      \\
    \end{array}
  \right.
\end{equation}
is well known to admit a unique weak entropy solution, as proved in
the classical result by Kru\v zkov~\cite[Theorem~5]{kruzkov}. The same
paper also provides the basic stability estimate on the dependence of
solutions from the initial data, see~\cite[Theorem~1]{kruzkov}. In the
same setting established in~\cite{kruzkov}, we provide here an
estimate on the dependence of the solutions to~(\ref{eq:probv}) from
the flow $f$, from the source $F$ and recover the known estimate on
the dependence from the initial datum $u_o$. A key intermediate result
is a bound on the total variation of the solution to~(\ref{eq:probv}),
which we provide in Theorem~\ref{teo:tv}.

In the case of a conservation law, i.e.~$F=0$, and with a flow $f$
independent from $t,x$, the dependence of the solution from $f$ was
already considered in~\cite{bouchutperthame}, where also other results
were presented. In this case, the $\tv$ bound is obvious, since $\tv
\left( u(t) \right) \leq \tv(u_o)$. The estimate provided
by~Theorem~\ref{teo:tv} slightly improves the analogous result
in~\cite[Theorem~3.1]{bouchutperthame} (that was already known,
see~\cite{dafermos, Lucier}), which reads (for a suitable absolute
constant $C$)
\begin{displaymath}
  \norma{u(t)-v(t)}_{\L1(\reali^N;\reali)}
  \leq
  \norma{u_o-v_o}_{\L1(\reali^N;\reali)} + C \, \tv(u_o) \, \Lip(f-g) \, t \,.
\end{displaymath}
Our result, given by Theorem~\ref{teo:estimates}, reduces to this
inequality when $f$ and $g$ are not dependent on $t,x$ and $F = G =
0$, but with $C=1$.

An flow dependent also on $x$ was considered in~\cite{chenkarlsen,
  KarlsenRisebro}, though in the special case $f(x,u) = l(x) \, g(u)$,
but with a source term containing a possibly degenerate parabolic
operator. There, estimates on the $\L1$ distance between solutions in
terms of the distance between the flows were obtained, but dependent
from an \emph{a priori} unknown bound on $\tv\left(u(t)\right)$.
Here, with no parabolic operators in the source term, we provide fully
explicit bounds both on $\tv \left( u(t) \right)$ and on the distance
between solutions.  Indeed, remark that with no specific assumptions
on the flow, $\tv \left( u(t) \right)$ may well blow up to $+\infty$
at $t = 0+$, as in the simple case $f(x,u) = \cos x$ with zero initial
datum.

Both the total variation and the stability estimates proved below turn
out to be optimal in some simple cases, in which optimal estimates are
known.

As an example of a possible application, we consider in
Section~\ref{sec:application} a toy model for a radiating gas. This
system was already considered in~\cite{ColomboGuerra, DiFrancesco,
  KN2, KN1, LattanzioMarcati, LinCoulombel, LiuTadmor, Serre2003}. It
consists of a balance law of the type~(\ref{eq:probv}), but with a
source that contains also a non local term, due to the convolution of
the unknown with a suitable kernel. Thanks to the present results, we
prove the well posedness of the model
extending~\cite[Theorem~2.4]{DiFrancesco} to more general flows,
sources and convolution kernels. Stability and total variation
estimates are also provided.

This paper is organized as follows: in Section~\ref{sec:Notation}, we
introduce the notation, state the main results and compare them with
those found in the literature. Section~\ref{sec:application} is
devoted to an application to a radiating gas model. Finally, in
sections~\ref{sec:prooftv} and~\ref{sec:proofcomparison} the detailed
proofs of theorems~\ref{teo:tv} and~\ref{teo:estimates} are provided.

\Section{Notation and Main Results}
\label{sec:Notation}

Denote $\rpic = \left[0, + \infty \right[$ and $\rpis = \left]0,
  +\infty \right[$. Below, $N$ is a positive integer, $\Omega = \rpis
\times \reali^N \times \reali$, $B(x,r)$ denotes the ball in
$\reali^N$ with center $x \in \reali^N$ and radius $r > 0$. The volume
of the unit ball $B(0,1)$ is $\omega_N$. For notational simplicity, we
set $\omega_0 = 1$.  The following relation can be proved using the
expression of $\omega_N$ in terms of the Wallis integral $W_N$:
\begin{equation}
  \label{eq:W}
  \frac{\omega_{N}}{\omega_{N-1}}
  =
  2 \, W_N
  \qquad \mbox{ where } \qquad
  W _N
  =
  \int_0^{\pi/2} (\cos \theta)^N \, \mathrm{d}\theta \,.
\end{equation}
In the present work, $\mathbf{1}_A$ is the characteristic function of
the set $A$ and $\delta_t$ is the Dirac measure centered at
$t$. Besides, for a vector valued function $f = f(x,u)$ with $u =
u(x)$, $\Div f$ stands for the total divergence. On the other hand,
$\div f$, respectively $\nabla f$, denotes the partial divergence,
respectively gradient, with respect to the space variables. Moreover,
$\partial_u$ and $\partial_t$ are the usual partial derivatives. Thus,
$\Div f = \div f + \partial_u f \cdot \nabla u$.

Recall the definition of weak entropy solution to~(\ref{eq:probv}),
see~\cite[Definition~1]{kruzkov}.

\begin{definition}
  \label{def:sol}
  A function $u \in \L\infty ( \rpis \times \reali^N; \reali)$ is a
  weak entropy solution to~(\ref{eq:probv}) if:

  \noindent 1.~for any constant $k \in \reali$ and any test function
  $\varphi \in \Cc{\infty}(\rpis\times\reali^N;\rpic)$
  \begin{equation}
    \label{eq:entropy}
    \begin{array}{l}
      \displaystyle
      \int_{\rpis} \! \int_{\reali^N} \!
      \left[
        (u-k) \, \dt \varphi
        +
        \left( f(t,x,u) - f(t,x,k) \right) \cdot \nabla \varphi
        +
        \left( F(t,x,u) - \div f(t,x,k) \right)\varphi
      \right]
      \\
      \displaystyle
      \qquad \qquad\qquad \qquad
      \times
      \mathrm{sign}(u-k) \,
      \mathrm{d}x \, \mathrm{d}t
      \geq
      0;
    \end{array}
    \!\!\!
  \end{equation}

  \noindent 2.~there exists a set $\mathcal{E}$ of zero measure in
  $\rpic$ such that for $t \in \rpic \setminus \mathcal{E}$ the
  function $u(t,x)$ is defined almost everywhere in $\reali^N$ and for
  any $r > 0$
  \begin{equation}
    \label{eq:data}
    \lim_{t \to 0,\, t \in \rpic \setminus \mathcal{E}}
    \int_{B(0,r)} \modulo{u(t,x) - u_o(x)} \mathrm{d}x
    =
    0 \,.
  \end{equation}
\end{definition}

\noindent Throughout this paper, we refer
to~\cite{AmbrosioFuscoPallara, volperthudjaev} as general references
for the theory of $\BV$~functions. In particular, recall the following
basic definition, see~\cite[Definition~3.4 and
Theorem~3.6]{AmbrosioFuscoPallara}.
\begin{definition}
  Let $u \in \Lloc1(\reali^N; \reali)$. Define
  \begin{eqnarray*}
    \tv(u)
    & = &
    \sup
    \left\{
      \int_{\reali^N} u \, \div \psi \, \mathrm{d}x
      \;\colon\;
      \psi\in\Cc1(\reali^N;\reali^N)
      \mbox{ and }
      \norma{\psi}_{\L\infty(\reali^N;\reali^N)} \le 1
    \right\}
    \\
    \BV (\reali^N;
    \reali)
    & = &
    \left\{
      u \in \Lloc1(\reali^N; \reali)
      \;\colon\;
      \tv (u) < +\infty
    \right\} .
  \end{eqnarray*}
\end{definition}
The following sets of assumptions will be of use below.
\begin{displaymath}
  \begin{array}{cl}
    \mathbf{(H1)}\!\!
    &
    \left\{
      \begin{array}{ll}
        f \in \C2(\Omega; \reali^N)
        &
        F \in \C1(\Omega; \reali)
        \\[5pt]
        \pt_u f\in \L\infty (\Omega; \reali^N)
        \\[5pt]
        \partial_u (F - \div f) \in \L\infty (\Omega; \reali)
        &
        F - \div f \in \L\infty (\Omega; \reali)
      \end{array}
    \right.
    \\[25pt]
    \mathbf{(H2)}\!\!
    &
    \left\{
      \begin{array}{ll}
        f \in \C2(\Omega; \reali^N)
        &
        F \in \C1(\Omega; \reali)
        \\[5pt]
        \nabla \pt_u f\in
        \L\infty(\Omega; \reali^{N\times N})
        &
        \displaystyle
        \int_{\rpis} \!\! \int_{\reali^N} \!
        \norma{\nabla(F - \div f)(t, x, \cdot)}_{\L\infty(\reali;\reali^{N})}
        \, \mathrm{d}x \, \mathrm{d}t
        <
        + \infty
        \\[10pt]
        \partial_t \partial_u f \in \L\infty (\Omega; \reali^N)
        &
        \partial_t F \in \L\infty (\Omega;\reali)
        \\[10pt]
        \partial_t \div f \in
        \L\infty (\Omega; \reali)
      \end{array}
    \right.
    \\[40pt]
    \mathbf{(H3)}\!\!
    &
    \left\{
      \begin{array}{l@{\qquad}l}
        f \in \C1(\Omega; \reali^N)
        &
        F \in \C0(\Omega; \reali)
        \qquad
        \partial_u F \in \L\infty(\Omega;\reali)
        \\[5pt]
        \pt_u f\in
        \L\infty(\Omega; \reali^N)
        &
        \displaystyle
        \int_{\rpis} \!\! \int_{\reali^N} \!
        \norma{(F- \div f)(t, x, \cdot)}_{\L\infty(\reali;\reali)} 
        \, \mathrm{d}x\, \mathrm{d}t < + \infty
      \end{array}
    \right.
  \end{array}
\end{displaymath}

\noindent The quantity $F-\div f$ has a particular role, since it
behaves as the \emph{``true''} source, see~(\ref{eq:Source}). We note
there that the assumptions above can be significantly softened in
various specific situations. For instance, the requirement that $f$ be
Lipschitz, which is however a standard hypothesis,
see~\cite[Paragraph~3]{bouchutperthame}, can be relaxed to $f$ locally
Lipschitz in the case $f=f(u)$ and $F = 0$, thanks to the maximum
principle~\cite[Theorem~3]{kruzkov}. Furthermore, the assumptions
above can be obviously weakened when aiming at estimates on bounded
time intervals.

Assumptions~\textbf{(H1)} are those used in the classical
results~\cite[Theorem~1 and Theorem~5]{kruzkov}. However, we stress
that the proofs below need less regularity. As in~\cite{kruzkov}, we
remark that no derivative of $f$ or $F$ in time is ever
needed. Furthermore, $f$ needs not be twice differentiable in $u$, for
the only second derivatives required are $\nabla_x \partial_u f$ and
$\nabla_x^2 f$.

We recall below the classical result by Kru\v zkov.

\begin{theorem}[Kru\v zkov]
  \label{teo:kruzkov}
  Let~\textbf{(H1)} hold. Then, for any $u_o \in \L\infty (\reali^N;
  \reali)$, there exists a unique weak entropy solution $u$
  to~(\ref{eq:probv}) in $\L\infty \left( \rpic; \Lloc{1}(\reali^N;
    \reali) \right)$ continuous from the right. Moreover, if a
  sequence $u_o^n \in \L\infty (\reali^N; \reali)$ converges to $u_o$
  in $\Lloc1$, then for all $t>0$ the corresponding solutions $u^n(t)$
  converge to $u(t)$ in $\Lloc1$.
\end{theorem}

\begin{remark}
  \label{rem:c}
  {\rm Under the conditions~\textbf{(H2)} and $\int_{\rpis} \!\!
    \int_{\reali^N} \!  \norma{(F- \div f)(t, x,
      \cdot)}_{\L\infty(\reali;\reali)} \, \mathrm{d}x\, \mathrm{d}t <
    + \infty$, see~\textbf{(H3)}, the estimate provided by
    Theorem~\ref{teo:tv} below, allows to use the technique described
    in~\cite[Theorem~4.3.1]{DafermosBook}, proving the continuity in
    time of the solution, so that $u \in \C0 \left( \rpic;
      \Lloc{1}(\reali^N; \reali) \right)$.}
\end{remark}

\subsection{Estimate on the Total Variation}

Recall that~\cite[Theorem~1.3]{KarlsenRisebro}
and~\cite[Theorem~3.2]{chenkarlsen} provide stability bounds
on~(\ref{eq:probv}), in the more general case with a degenerate
parabolic source, but assuming \emph{a priori} bounds on the total
variation of solutions. Our first result provides these bounds.

\begin{theorem}
  \label{teo:tv}
  Assume that~\textbf{(H1)} and~\textbf{(H2)} hold. Let $u_o \in
  \BV(\reali^N; \reali)$. Then, the weak entropy solution $u$
  of~(\ref{eq:probv}) satisfies $u(t) \in \BV(\reali^N; \reali)$ for
  all $t > 0$. Moreover, let
  \begin{equation}
    \label{eq:kappao}
    \kappa_o
    =
    N \, W_N
    \left(
      (2N+1) \, \norma{\nabla \, \partial_u
        f}_{\L\infty(\Omega;\reali^{N\times N})}
      +
      \norma{\partial_u F}_{\L\infty(\Omega;\reali)}
    \right)
  \end{equation}
  with $W_N$ as in~(\ref{eq:W}). Then, for all $T > 0$,
  \begin{equation}
    \label{result}
    \tv \left( u(T) \right)
    \leq
    \displaystyle
    \tv(u_o) \, e^{\kappa_o T}
    \displaystyle
    +
    N W_N \!\! \int_0^T e^{\kappa_o(T-t)} \int_{\reali^N}
    \norma{\nabla( F - \div f)(t, x, \cdot)}_{\L\infty}
    \mathrm{d}x\, \mathrm{d}t \,.
  \end{equation}
\end{theorem}

\noindent This estimate is optimal in the following situations:
\begin{enumerate}
\item If $f$ is independent from $x$ and $F = 0$, then $\kappa_o =0$
  and the integrand in the right hand side above vanishes. Hence,
  (\ref{result}) reduces to the well known optimal bound $\tv \left(
    u(t) \right) \leq \, \tv (u_o)$.
\item In the 1D case, if $f$ and $F$ are both independent from $t$ and
  $u$, then $\kappa_o =0$ and~(\ref{eq:probv}) reduces to the ordinary
  differential equation $\partial_t u = F - \div f$. In this
  case,~(\ref{result}) becomes
  \begin{equation}
    \label{eq:Source}
    \tv \left( u(t) \right) \leq  \tv (u_o) + t \, \tv(F - \div f) \,.
  \end{equation}
\item If $f = 0$ and $F = F(t)$ then, trivially, $\tv \left( u(t)
  \right) = \tv (u_o)$ and~(\ref{result}) is optimal.
\end{enumerate}
\noindent A simpler but slightly weaker form of~(\ref{result}) is
\begin{displaymath}
  \tv \left( u(T) \right)
  \leq
  \displaystyle
  \tv(u_o) \, e^{\kappa_o T}
  \displaystyle
  +
  N W_N  \frac{e^{\kappa_o T}-1}{\kappa_o} \sup_{t \in [0,T]}\int_{\reali^N}
  \norma{\nabla( F - \div f)(t, x, \cdot)}_{\L\infty}
  \mathrm{d}x
\end{displaymath}
when the right hand side is bounded.

\subsection{Stability of Solutions with Respect to Flow and Source}

Consider now~(\ref{eq:probv}) together with the analogous problem
\begin{equation}
  \label{eq:probu}
  \left\{\begin{array}{l@{\qquad}rcl}
      \partial_t v + \Div \, g(t,x,v) = G(t,x,v) 
      & (t,x) & \in & \rpis \times \reali^N
      \\
      v(0,x)=v_o(x) 
      & x & \in & \reali^N \,.
    \end{array}
  \right.
\end{equation}
We aim at estimates for the difference $u - v$ between the solutions
in terms of $f - g$, $F - G$ and $u_o - v_o$. Estimates of this type
were derived by Bouchut \& Perthame in~\cite{bouchutperthame} when
$f$, $g$ depend only on $u$ and $F = G = 0$. Here, we generalize their
result adding the $(t,x)$-dependence. The present technique is
essentially based on Theorem~\ref{teo:tv}.

\begin{theorem}
  \label{teo:estimates}
  Let $(f, F)$, $(g,G)$ verify~\textbf{(H1)}, $(f,F)$ verify
  \textbf{(H2)} and $(f-g,F-G)$ verify~\textbf{(H3)}. Let $u_o, v_o
  \in \BV(\reali^N; \reali)$. We denote $\kappa_o$ as
  in~(\ref{eq:kappao}) and introduce
  \begin{displaymath}
    \kappa
    =
    2N \norma{\nabla\pt_u f }_{\L{\infty}(\Omega ; \reali^{N\times N})}
    +
    \norma{\pt_u F}_{\L{\infty}(\Omega ; \reali)}
    +
    \norma{\pt_u (F-G)}_{\L\infty(\Omega;\reali)} 
    \mbox{ and } 
    M
    =
    \norma{\partial_u g}_{\L\infty(\Omega;\reali^N)}.
  \end{displaymath}
  Then, for any $T,R >0$ and $x_o \in \reali^N$, the following
  estimate holds:
  \begin{eqnarray*}
    & &
    \int_{\norma{x-x_o}\le R} \modulo{u(T,x)-v(T,x)}\mathrm{d}x
    \; \leq \;
    e^{\kappa T}
    \int_{\norma{x-x_o}\leq R + M T} \modulo{u_o(x) - v_o(x)} 
    \, \mathrm{d}x
    \\
    & + &
    \frac{e^{\kappa_o T}-e^{\kappa T}}{\kappa_o-\kappa}  \, \tv(u_o) \, 
    \norma{\pt_u(f-g)}_{\L{\infty}}  
    \\
    & + &
    NW_N 
    \left(
      \int_0^T \frac{e^{\kappa_o (T-t)}-e^{\kappa (T-t)}}{\kappa_o-\kappa} 
      \int_{\reali^N} \norma{\nabla(F-\div f)(t, x,\cdot)}_{\L\infty}
      \mathrm{d}x \, \mathrm{d}t
    \right)
    \norma{\pt_u(f-g)}_{\L{\infty}}
    \\
    & + &
    \int_0^T e^{\kappa (T-t)} \int_{\norma{x-x_o}\leq R+M(T-t)}
    \norma{\left((F-G) - \div(f-g) \right)(t,x,\cdot)}_{\L{\infty}} 
    \, \mathrm{d}x\, \mathrm{d}t \,.
  \end{eqnarray*}
\end{theorem}

\noindent The above inequality is undefined for $\kappa = \kappa_o$
and, in this case, it reduces to~(\ref{eqKKo}).  This bound is optimal
in the following situations, where $u_o,v_o \in \L1(\reali^N;\reali)$.
\begin{enumerate}
\item In the standard case of a conservation law, i.e.~when $F = G =0$
  and $f,g$ are independent of $x$, we have $\kappa_o = \kappa =0$ and
  the result of Theorem~\ref{teo:estimates} becomes,
  see~\cite[Theorem~2.1]{BianchiniColombo},
  \begin{displaymath}
    \norma{u(T) - v(T)}_{\L1(\reali^N;\reali)}
    \leq
    \norma{u_o - v_o}_{\L1(\reali^n;\reali)}
    +
    T \; \tv (u_o) \, \norma{\partial_u (f - g)}_{\L\infty(\Omega;\reali^N)} \,.
  \end{displaymath}
\item If $\partial_u f = \partial_u g =0$ and $\partial_u F
  = \partial_u G = 0$, then $\kappa_o = \kappa = 0$ and
  Theorem~\ref{teo:estimates} now reads
  \begin{displaymath}
    \!\!\!\!\!\!\!\!\!\!\!\!
    \norma{u(T) -  v(T)}_{\L1(\reali^N;\reali)} 
    \leq 
    \norma{u_o - v_o}_{\L1(\reali^N;\reali)} 
    + \!
    \int_0^T \!\!
    \norma{\left[(F - G) - \div (f - g)\right](t)}_{\L1(\reali^N;\reali)}
    \mathrm{d}t.
  \end{displaymath}
\item If $(f,F)$ and $(g,G)$ are dependent only on $x$, then
  Theorem~\ref{teo:estimates} reduces to
  \begin{displaymath}
    \norma{u(T) - v(T)}_{\L1(\reali^N;\reali)} 
    \leq 
    \norma{u_o - v_o}_{\L1(\reali^N;\reali)} 
    + 
    T \, \norma{(F-G) -\div(f-g)}_{\L1(\reali^N;\reali)} \,.
  \end{displaymath}
\end{enumerate}

\noindent The estimate obtained in Theorem~\ref{teo:estimates} shows
also that, depending on the properties of specific applications, the
regularity requirement $f \in \C2(\Omega; \reali^N)$ can be
significantly relaxed. For instance, in the case $f(t,x,u) = q(u) \,
v(x)$ considered in~\cite{chenkarlsen, KarlsenRisebro}, asking $q$ of
class $\C1$ and $v$ of class $\C2$ is sufficient. See also
Section~\ref{sec:application} for a case in which the required
regularity in time can be reduced.

In the case of conservations laws, i.e.~when $F = G = 0$, one proves
that $\kappa < \kappa_o$ and the estimate in
Theorem~\ref{teo:estimates} takes the somewhat simpler form
\begin{eqnarray*}
  & &
  \int_{\norma{x-x_o}\le R} \modulo{u(T,x)-v(T,x)}\mathrm{d}x
  \; \leq \;
  e^{\kappa T}
  \int_{\norma{x-x_o}\leq R + M T} \modulo{u_o(x) - v_o(x)} 
  \, \mathrm{d}x
  \\[5pt]
  & + &
  T \, e^{\kappa_oT} \, \tv(u_o) \, 
  \norma{\pt_u(f-g)}_{\L{\infty}}  
  \\
  & + &
  NW_N \, T^2 \, e^{\kappa_oT}
  \sup_{t \in [0,T]} 
  \left(
    \int_{\reali^N} \norma{\nabla \div f(t, x,\cdot)}_{\L\infty}
    \mathrm{d}x
  \right)
  \; \norma{\pt_u(f-g)}_{\L{\infty}}
  \\
  & + &
  T e^{\kappa_oT}
  \sup_{t \in [0,T]}
  \int_{\norma{x-x_o} \leq R+M(T-t)}
  \norma{\div(f-g) (t,x,\cdot)}_{\L{\infty}} 
  \, \mathrm{d}x
\end{eqnarray*}
when the right hand side is bounded.  In the case considered
in~\cite[Theorem~3.1]{bouchutperthame}, $f=f(u)$, $\kappa_o=0$ and we
obtain~\cite[formula~(3.2)]{bouchutperthame} with $1$ instead of the
constant $C$ therein.

\Section{Application to a Radiating Gas Model}
\label{sec:application}

The following balance law is a toy model inspired by Euler equations
for radiating gases:
\begin{equation}
  \label{eq:BL}
  \partial_t u + \Div f(t,x,u) = -u + K \ast_x u \,.
\end{equation}
It has been extensively studied in the literature when $f=f(u)$, see
for instance~\cite{KN2, KN1, LattanzioMarcati, LiuTadmor, Serre2003}
for the scalar $1$D case, \cite{ColomboGuerra, LinCoulombel} for 1D
systems, \cite{DiFrancesco} for the scalar $N$D case.

The estimate provided by Theorem~\ref{teo:estimates} allows us to
present an alternative proof of the well posedness of~(\ref{eq:BL})
proved in~\cite{DiFrancesco}. Furthermore, we add stability estimates
on the dependence of the solution from $f$ and $K$, in the case of $f$
dependent also on $t,x$ and with more general source terms.

\begin{theorem}
  \label{thm:appl}
  Let $(f,F)$ satisfy~\textbf{(H1)}, \textbf{(H2)}
  and~\textbf{(H3)}. Assume that
  \begin{displaymath}
    \mbox{{\rm\textbf{(K)}}} \qquad
    K
    \in (\C2 \cap \L\infty)(\rpis \times \reali^N; \reali)
    \quad \mbox{ and } \quad
    K
    \in 
    \L\infty \left(\rpis; \W{2}{1}(\reali^N;\reali) \right).
  \end{displaymath}
  Then, for any $u_o \in (\BV\cap\L1)(\reali^N;\reali)$, the Cauchy
  problem
  \begin{equation}
    \label{eq:toy}
    \left\{\begin{array}{l@{\qquad}rcl}
        \partial_t u + \Div  f(t,x,u) = F(t,x,u) + K \ast_x u
        & (t,x) & \in & \rpis\times\reali^N
        \\
        u(0,x)=u_o(x) & x & \in & \reali^N
        \\
      \end{array}
    \right.
  \end{equation}
  admits a unique weak entropy solution $u \in \C0 \left( \rpic; \L1
    (\reali^N; \reali) \right)$. Moreover, denoting $k =
  \norma{K}_{\L\infty(\rpis;\L1(\reali^N;\reali))}$, for all $T > 0$,
  the following estimate holds:
  \begin{eqnarray*}
    \tv \left( u(T) \right)
    & \leq &
    e^{(\kappa_o+N W_N k) T} \; \tv(u_o)
    \\
    & &
    + 
    NW_N \int_0^T e^{(\kappa_o+ N W_N k)(T-t)}
    \int_{\reali^N} 
    \norma{\nabla\left(F-\div f \right)(t,x,\cdot)}_{\L\infty}
    \, \mathrm{d}x\, \mathrm{d}t.
  \end{eqnarray*}
  If $F(t,x,0) - \div f(t,x,0) = 0$ for all $t \in [0,T]$ and $x \in
  \reali^N$, then
  \begin{enumerate}
  \item $ \displaystyle \norma{u(T)}_{\L1(\reali^N;\reali)} \leq
    e^{(\kappa +k)T} \norma{u_o}_{\L1(\reali^N;\reali)}$.
  \item Let $\tilde K$ satisfy~\textbf{(K)} and call $\tilde u$ the
    solution to~(\ref{eq:toy}) with $K$ replaced by $\tilde K$. Then,
    \begin{equation}
      \label{eq:end}
      \norma{u(T)-\tilde u(T)}_{\L1(\reali^N;\reali)}
      \leq
      \norma{u_o}_{\L1(\reali^N;\reali)} \;
      \frac{e^{kT}-e^{\tilde k T}}{k - \tilde k} 
      \; \norma{K-\tilde K}_{\L\infty(\rpis;\L1(\reali^N;\reali))} \,.
    \end{equation}
  \end{enumerate}

\end{theorem}

\begin{proof} %{Theorem~\ref{thm:appl}}
  Fix a positive $T$ (to be specified below) and consider the Banach
  space $X = \C0 \left( [0,T]; \L1(\reali^N;\reali) \right)$ equipped
  with the usual norm $\norma{u}_X =
  \norma{u}_{\L\infty(\rpis;\L1(\reali^N;\reali))}$. Define on $X$ the
  map ${\mathcal T}$ so that ${\mathcal T}(w) =u$ if and only if $u$
  solves
  \begin{equation}
    \label{eq:weak}
    \left\{\begin{array}{l@{\qquad}rcl}
        \partial_t u + \Div  f(t,x,u) = F(t,x,u) + K \ast_x w
        & (t,x) & \in & \rpis\times\reali^N
        \\
        u(0,x)=u_o(x) & x & \in & \reali^N
        \\
      \end{array}
    \right.
  \end{equation}
  in the sense of Definition~\ref{def:sol}. Note that the source term
  does not have the regularity required in~\textbf{(H1)}. However, by
  the estimate in Theorem~\ref{teo:estimates}, we can prove
  that~(\ref{eq:weak}) does indeed have a unique weak entropy
  solution, see Lemma~\ref{lem:weak} for the details.  The fixed
  points of $\mathcal T$ are the solutions to~(\ref{eq:BL}). By
  Theorem~\ref{teo:kruzkov} and Remark~\ref{rem:c}, ${\mathcal T}w \in
  X$ for all $w \in X$. We now show that ${\mathcal T}$ is a
  contraction, provided $T$ is sufficiently small.  Note that
  \begin{eqnarray*}
    \kappa_o
    & = &
    N \, W_N
    \left(
      (2N+1) \, \norma{\nabla \, \partial_u f}_{\L\infty}
      +
      \norma{\partial_u F}_{\L\infty}
    \right)
    \\
    \kappa 
    & = &
    2 N \norma{\nabla \partial_u f}_{\L\infty}
    +
    \norma{\partial_u F}_{\L\infty} \,.
  \end{eqnarray*}
  Moreover, by Theorem~\ref{teo:estimates}
  \begin{eqnarray*}
    d ({\mathcal T}w_1, {\mathcal T}w_2)
    & = &
    \sup_{t \in [0,T]} \norma{{\mathcal T}w_1 - {\mathcal T}w_2}_{\L1}
    \\
    & \leq &
    \sup_{t \in [0,T]} 
    \left(
      \frac{e^{\kappa t} -1}{\kappa}
      \sup_{\tau \in [0,t]} 
      \norma{K(\tau) \ast_x (w_1 - w_2) (\tau)}_{\L1}
    \right)
    \\
    & \leq &
    \frac{e^{\kappa T} -1}{\kappa}
    \sup_{\tau \in [0,T]} 
    \norma{K(\tau)}_{\L1} \, \norma{(w_1 - w_2) (\tau)}_{\L1}
    \\
    & \leq &
    \frac{e^{\kappa T} -1}{\kappa} \, k \, d(w_1, w_2) \,.
  \end{eqnarray*}
  Therefore, ${\mathcal T}$ is a contraction as soon as $T$ is smaller
  than a threshold that depends only on $\norma{\partial_u
    F}_{\L\infty(\Omega:\reali)}, \norma{\nabla \partial_u
    f}_{\L\infty(\Omega:\reali^{N\times N})}$ and on
  $\norma{K}_{\L\infty(\rpis;\L1(\reali^N;\reali))}$. Therefore, we
  proved the well posedness of~(\ref{eq:toy}) globally in time.

  Consider the bound on $\tv \left(u(t)\right)$. By
  Theorem~\ref{teo:tv},
  \begin{eqnarray*}
    \tv \left( u(T) \right)
    & \leq &
    \tv (u_o)
    +
    N W_N
    \int_0^T e^{\kappa_o(T-t)}
    \int_{\reali^N}
    \norma{\nabla(F-\div f)(t,x,\cdot)}_{\L\infty(\reali;\reali^N)}
    \, \mathrm{d}x\, \mathrm{d}t
    \\
    & &
    + 
    N W_N
    \int_0^T e^{\kappa_o(T-t)}
    k \, \tv \left( u(t) \right) \, \mathrm{d}t
  \end{eqnarray*}
  and an application of Gronwall Lemma gives the desired bound.

  We estimate the $\L1$ norm of the solution to~(\ref{eq:toy}),
  comparing it with the solution to
  \begin{equation}
    \label{eq:toy0}
    \left\{\begin{array}{l@{\qquad}rcl}
        \partial_t u + \Div  f(t,x,u) = F(t,x,u) + K \ast_x u
        & (t,x) & \in & \rpis\times\reali^N
        \\
        u(0,x)=0 & x & \in & \reali^N.
        \\
      \end{array}
    \right.
  \end{equation}
  By assumption, $0$ solves~(\ref{eq:toy0}), hence it is its unique
  solution. Then, evaluating the distance between the solutions
  of~(\ref{eq:toy}) and~(\ref{eq:toy0}) by means of
  Theorem~\ref{teo:estimates}, we get
  \begin{eqnarray*}
    e^{-\kappa T}\norma{u(T)}_{\L1(\reali^N;\reali)}
    & \leq &
    \norma{u_o}_{\L1(\reali^N;\reali)}
    +
    \int_0^T e^{-\kappa t} \int_{\reali^N} \modulo{K\ast_x u (t,x)}
    \, \mathrm{d}x\, \mathrm{d}t
  \end{eqnarray*} 
  and, thanks to Gronwall Lemma, we obtain:
  \begin{displaymath}
    \norma{u(T)}_{\L1(\reali^N;\reali)}
    \leq 
    e^{\left(\kappa+k\right) T}\norma{u_o}_{\L1(\reali^N;\reali)} \,.
  \end{displaymath}
  The final estimate~(\ref{eq:end}) follows from
  Theorem~\ref{teo:estimates}:
  \begin{eqnarray*}
    \!\!\!\!\!\!
    & &
    e^{-\kappa T} \norma{(u-\tilde u)(T)}_{\L1(\reali^N;\reali)}
    \\
    \!\!\!\!\!\!
    & \leq & 
    \norma{K - \tilde K}_{\L\infty(\rpis;\L1(\reali^N;\reali))}
    \int_0^T e^{-\kappa t} \norma{u(t)}_{\L1(\reali^N;\reali)}
    \, \mathrm{d}t 
    + 
    k \!\! \int_0^T e^{-\kappa t} \norma{(u-\tilde u)(t)}_{\L1(\reali^N;\reali)} 
    \,\mathrm{d}t
    \\
    \!\!\!\!\!\!
    & \leq &
    \norma{K - \tilde K}_{\L\infty(\rpis;\L1(\reali^N;\reali))}
    \, \norma{u_o}_{\L1(\reali^N;\reali)}  \frac{e^{kT}-1}{k} 
    + 
    \tilde k \int_0^T e^{-\kappa t} \norma{(u-\tilde u)(t)}_{\L1(\reali^N;\reali)}
    \,\mathrm{d}t
  \end{eqnarray*}
  and thanks to Gronwall Lemma, we get the result.

  The continuity in time is proved as described in Remark~\ref{rem:c}.
\end{proof}

\begin{lemma}
  \label{lem:weak}
  Let $f,F$ satisfy~\textbf{(H1)} and $K$ satisfy~\textbf{(K)}. If $w
  \in \L\infty(\rpic \times \reali^N;\reali)$, then the estimates in
  Theorem~\ref{teo:tv} and in Theorem~\ref{teo:estimates} apply also
  to~(\ref{eq:weak}).
\end{lemma}

\begin{proof}
  Fix positive $T,R$ and let $w_n$ be a sequence of $\C\infty$
  functions converging to $w$ in $\L1 \left( [0,T]\times
    \reali^N;\reali \right)$. Apply Theorem~\ref{teo:kruzkov} to the
  approximate problem
  \begin{equation}
    \label{eq:weakn}
    \left\{\begin{array}{l@{\qquad}rcl}
        \partial_t u + \Div  f(t,x,u) = F(t,x,u) + K \ast_x w_n
        & (t,x) & \in & \rpis\times\reali^N
        \\
        u(0,x)=u_o(x) & x & \in & \reali^N
        \\
      \end{array}
    \right.
  \end{equation}
  to ensure the existence of its weak entropy solution $u_n$. Apply
  Theorem~\ref{teo:estimates} to estimate the distance between $u_n$
  and $u_{n-1}$:
  \begin{eqnarray*}
    \norma{u_n - u_{n-1}}_{\L\infty([0,T];\L1(\reali^N;\reali))}
    & \leq &
    \int_0^T e^{\kappa(T-t)} \int_{\reali^N} \modulo{K * (w_n-w_{n-1}) (t,x)} 
    \, \mathrm{d}x \, \mathrm{d}t
    \\
    & \leq &
    e^{\kappa T} \, k \,
    \norma{w_n - w_{n-1}}_{\L1([0,T]\times\reali^N;\reali)} 
  \end{eqnarray*}
  showing that the $u_n$ form a Cauchy sequence. Their limit $u$
  solves~(\ref{eq:toy}), as it follows passing to the limit over $n$
  in the integral conditions~(\ref{eq:entropy})--(\ref{eq:data}) and
  applying the Dominated Convergence Theorem. The estimates in
  theorems~\ref{teo:tv} and~\ref{teo:estimates} are extended
  similarly.
\end{proof}

\Section{Proof of Theorem~\ref{teo:tv} }
\label{sec:prooftv}

\begin{lemma}
  Fix a function $\mu_1 \in \Cc\infty(\rpic;\rpic)$ with
  \begin{equation}
    \label{eq:mu}
    \mathrm{supp}(\mu_1) \subseteq \left[0, 1 \right[
    ,\quad
    \int_{\rpis}
    r^{N-1} \mu_1(r) \, \mathrm{d}r = \frac{1}{N \omega_N}
    ,\quad
    \mu_1' \leq 0
    ,\quad
    \mu_1^{(n)}(0) = 0 \mbox{ for } n\geq 1.
  \end{equation}
  Define
  \begin{equation}
    \label{eq:Mu}
    \mu(x) = \frac{1}{\gd^N} \, \mu_1 \left(
      \frac{\norma{x}}{\gd} \right) \,.
  \end{equation}
  Then, recalling that $\omega_0 = 1$,
  \begin{eqnarray}
    \label{eq:mu1}
    \int_{\reali^N} \mu(x) \,\mathrm{d} x
    & = &
    1 \, ,
    \\
    \label{eq:mu2}
    \int_{\reali^N}
    \modulo{x_1} \, \mu_1 \left( \norma{x}\right) \, \mathrm{d}x
    & = &
    \frac{2}{N} \, \frac{\omega_{N-1}}{\omega_N} \,
    \int_{\reali^N}
    \norma{x} \, \mu_1 \left( \norma{x}\right) \, \mathrm{d}x\, ,
    \\
    \label{eq:mu3}
    \int_{\reali^N}
    \norma{x} \, \norma{\nabla \mu (x)} \, \mathrm{d}x
    & = &
    - \int_{\reali^N}
    \norma{x} \, \mu_1' \left( \norma{x}\right) \, \mathrm{d}x
    \;\, = \;\,
    N\,  ,
    \\
    \label{eq:mu4}
    \int_{\reali^N}
    \norma{x}^2 \, \mu_1' \left( \norma{x} \right) \, \mathrm{d}x
    & = &
    - (N+1) \,
    \int_{\reali^N}
    \norma{x} \, \mu_1 \left( \norma{x}\right) \, \mathrm{d}x \,.
  \end{eqnarray}
\end{lemma}

\begin{proof}
  The first relation is immediate. Equalities~(\ref{eq:mu3})
  and~(\ref{eq:mu4}) follow directly from an integration by
  parts. Consider~(\ref{eq:mu2}). The cases $N = 1,2,3$ follow from
  direct computations. Let $N \geq 4$ and pass to spherical
  coordinates $(\rho, \theta_1 , \ldots, \theta_{N-1})$,
  \begin{eqnarray*}
    x_1 & = & \rho \, \cos \theta_{N-1}
    \\
    x_2 & = & \rho \, \sin \theta_{N-1} \, \cos \theta_{N-2}
    \\
    \vdots & \vdots & \vdots
    \\
    x_{N-1} & = & \rho \, \sin \theta_{N-1} \, \sin \theta_{N-2} \cdots \cos \theta_1
    \\
    x_N & = & \rho \, \sin \theta_{N-1} \, \sin \theta_{N-2} \cdots \sin \theta_1
  \end{eqnarray*}
  with $\rho \in \rpis$, $\theta_1 \in \left[0, 2\pi\right[$ and
  $\theta_j \in [0,\pi]$ for $j=2, \ldots, N-1$. If $N \geq 4$
  % Compute preliminarily, for $N \geq 2$,
  % \begin{eqnarray*}
  %   \int_0^\pi \modulo{\cos \theta} \, (\sin \theta)^{N-2} \,
  %   \mathrm{d}\theta & = & \frac{2}{N-1}
  %   \\
  %   \int_{\rpis} \rho^N \, \mu_1(\rho) \, \mathrm{d}\rho & = &
  %   \frac{1}{N\omega_N} \int_{\reali^N} \norma{x} \,
  %   \mu_1\left(\norma{x}\right) \, \mathrm{d}x
  % \end{eqnarray*}
  % and, for $N \geq 3$,
  % \begin{displaymath}
  %   \int_0^{2\pi} \int_0^\pi \cdots \int_0^\pi \left(
  %     \prod_{j=2}^{N-2} (\sin \theta_j)^{j-1} \right)
  %   \mathrm{d}\theta_{N-1} \, \mathrm{d} \theta_{N-2} \cdots
  %   \mathrm{d}\theta_1 \, \mathrm{d}\rho = (N-1) \omega_{N-1} \,.
  % \end{displaymath}
  % Hence
  % \begin{eqnarray*}
  %   & & \int_{\reali^N} \modulo{x_1} \, \mu_1 \left(
  %     \norma{x}\right) \, \mathrm{d}x
  %   \\
  %   & = & \int_0^{2\pi} \int_0^\pi \cdots \int_0^\pi \modulo{\cos
  %     \theta_{N-1}} \, \rho^N \, \mu_1(\rho) \, \left(
  %     \prod_{j=2}^{N-1} (\sin \theta_j)^{j-1} \right)
  %   \mathrm{d}\theta_{N-1} \, \mathrm{d} \theta_{N-2} \cdots
  %   \mathrm{d}\theta_1
  %   \\
  %   & & \quad \times \int_{\rpis} \rho^N \, \mu_1(\rho) \,
  %   \mathrm{d}\rho
  % \end{eqnarray*}
  % \newpage
  \begin{eqnarray*}
    & &
    \int_{\reali^N}
    \modulo{x_1} \, \mu_1 \left( \norma{x}\right) \, \mathrm{d}x
    \\
    & = &
    \int_{\rpis} \int_0^{2\pi} \int_0^\pi \cdots \int_0^\pi
    \modulo{\cos \theta_{N-1}} \, \rho^N \, \mu_1(\rho) \,
    \left( \prod_{j=2}^{N-1} (\sin \theta_j)^{j-1} \right)
    \mathrm{d}\theta_{N-1} \, \mathrm{d} \theta_{N-2} \cdots
    \mathrm{d}\theta_1
    \, \mathrm{d}\rho
    \\
    & = &
    \int_0^{2\pi} \int_0^\pi \cdots \int_0^\pi
    \left( \prod_{j=2}^{N-2} (\sin \theta_j)^{j-1} \right)
    \mathrm{d} \theta_{N-2} \cdots \mathrm{d}\theta_1
    \\
    & &
    \qquad
    \times
    \left(
      \int_0^\pi \modulo{\cos \theta_{N-1}} \left(\sin \theta_{N-1}\right)^{N-2} \,
      \mathrm{d}\theta_{N-1}
    \right)
    \int_{\rpis} \rho^N \, \mu_1(\rho) \, \mathrm{d}\rho
    \\
    & = &
    (N-1) \omega_{N-1} \frac{2}{N-1} \frac{1}{N\omega_N}
    \int_{\reali^N}
    \norma{x} \, \mu_1 \left( \norma{x}\right) \, \mathrm{d}x
    \\
    & = &
    \frac{2}{N} \, \frac{\omega_{N-1}}{\omega_N} \,
    \int_{\reali^N}
    \norma{x} \, \mu_1 \left( \norma{x}\right) \, \mathrm{d}x
  \end{eqnarray*}
  completing the proof.
\end{proof}

Recall the following theorem (see~\cite[Theorem~3.9 and
Remark~3.10]{AmbrosioFuscoPallara}):

\begin{theorem}
  \label{thm:AFP}
  Let $u \in \Lloc1(\reali^N;\reali)$, then $u \in \BV(
  \reali^N;\reali)$ if and only if there exists a sequence $u_n$ in
  $\C\infty (\reali^N;\reali)$ converging to $u$ in $\Lloc1$ and
  satisfying
  \begin{displaymath}
    \lim_{n\to +\infty} \int_{\reali^N} \norma{\nabla u_n (x)}\, \mathrm{d}x
    =
    L \quad \mbox{ with } \quad
    L < \infty \,.
  \end{displaymath}
  Moreover, $\tv(u)$ is the least constant $L$ for which there exists
  a sequence as above.
\end{theorem}

\begin{proposition}
  \label{prop:normtv}
  Fix $\mu_1$ as in~(\ref{eq:mu}). Let $u \in \Lloc1(\reali^N;
  \reali)$ admit a constant $\tilde C$ such that for all positive
  $\gd$, $R$ and with $\mu$ as in~(\ref{eq:Mu})
  \begin{equation}
    \label{ineq}
    \frac{1}{\gd} \int_{\reali^N} \int_{B(x_o,R)}
    \modulo{u(x)-u(x-z)} \, \mu(z)
    \, \mathrm{d}x \, \mathrm{d}z
    \leq
    \tilde C.
  \end{equation}
  Then, $u \in \BV(\reali^N; \reali)$ and $\tv(u) \leq \tilde C /
  C_1$, where
  \begin{equation}
    \label{eq:C1}
    C_1
    =
    \int_{\reali^N} \modulo{x_1} \, \mu_1\left(\norma{x} \right) 
    \, \mathrm{d}x \,.
  \end{equation}
  Note that $C_1 \in \left]0, 1 \right[$. If moreover $u \in
  \C1(\reali^N;\reali)$, then
  \begin{equation}
    \label{eq:TV1}
    \tv (u)
    =
    \frac{1}{C_1}
    \lim_{\gd \to 0}
    \frac{1}{\gd} \int_{\reali^N} \int_{\reali^N}
    \modulo{u(x)-u(x-z)} \, \mu(z)
    \, \mathrm{d}x \, \mathrm{d}z \,.
  \end{equation}
\end{proposition}

\begin{proof}
  We introduce now a regularisation of $u$: $u_h = u \ast \mu_h$, with
  $\mu_h(x) = \mu_1 \left( \norma{x}/h \right) / h^N$. Note that $u_h
  \in \C\infty (\reali^N; \reali)$ and $u_h$ converges to $u$ in
  $\Lloc1$ as $h \to 0$. Furthermore, for $R$ and $h$ positive, we
  have
  \begin{eqnarray*}
    & &
    \frac{1}{\gd}
    \int_{\reali^N} \int_{B(x_o,R)}
    \modulo{u_h(x)-u_h(x-z)} \, \mu(z)
    \, \mathrm{d}x \, \mathrm{d}z
    \\
    & \leq &
    \frac{1}{\gd}
    \int_{\reali^N} \int_{B(x_o,R+h)}
    \modulo{u(x)-u(x-z)} \, \mu(z)
    \, \mathrm{d}x \, \mathrm{d}z
    \\
    & \leq &
    \tilde C
  \end{eqnarray*}
  and
  \begin{displaymath}
    \frac{u_h(x) - u_h(x-\gd z)}{\gd}
    =
    \int_0^1 \nabla u_h(x-\gd s z) \cdot z \,\mathrm{d}s \,.
  \end{displaymath}
  Thanks to the Dominated Convergence Theorem, at the limit $\gd\to 0$
  we get
  \begin{displaymath}
    \int_{\reali^N} \int_{B(x_o, R)}
    \modulo{\nabla u_h(x) \cdot z} \, \mu_1(\norma{z})
    \, \mathrm{d}x \, \mathrm{d}z
    \leq
    \tilde C \,.
  \end{displaymath}

  Remark that for fixed $x\in B(x_o, R)$, when $\nabla u_h(x) \neq 0$,
  the scalar product $\nabla u_h(x)\cdot z$ is positive (respectively,
  negative) when $z$ is in a half-space, say $H_x^+$ (respectively,
  $H_x^-$). We can write $z = \alpha \frac{\nabla
    u_h(x)}{\norma{\nabla u_h(x)}} + w$, with $\alpha \in \reali$ and
  $w$ in the hyperplane $H^o_x = \nabla u_h(x)^\perp$. Hence
  \begin{eqnarray*}
    \int_{\reali^N}
    \modulo{\nabla u_h(x) \cdot z} \, \mu_1(\norma{z}) \, \mathrm{d}z
    & = &
    \int_{H^+_x} \!\!
    \nabla u_h(x) \cdot z \, \mu_1(\norma{z}) \,
    \mathrm{d}z
    +
    \int_{H^-_x} \!\!
    \nabla u_h(x) \cdot (-z) \, \mu_1(\norma{z}) \,
    \mathrm{d}z
    \\
    & = &
    2 \int_{H^+_x}
    \nabla u_h(x) \cdot z \, \mu_1(\norma{z}) \, \mathrm{d}z
    \\
    & = &
    2 \int_{\rpis} \int_{H^o_x}
    \alpha \, \norma{\nabla u_h(x)} \,
    \mu_1(\sqrt{\alpha^2+\norma{w}^2})
    \, \mathrm{d}w \, \mathrm{d}\alpha
    \\
    & = &   \int_{\reali} \int_{H^o_x}
    \modulo{\alpha} \, \norma{\nabla u_h(x)} \,
    \mu_1(\sqrt{\alpha^2+\norma{w}^2})
    \, \mathrm{d}w \, \mathrm{d}\alpha
    \\
    & = & \norma{\nabla u_h(x)} \int_{\reali^N} \modulo{z_1} \,
    \mu_1(\norma{z}) \, \mathrm{d}z\, .
  \end{eqnarray*}
  Define $C_1$ as in~(\ref{eq:C1}) and note that $C_1 \in \left]0,
    1\right[$. Then we obtain, for all $R>0$,
  \begin{equation}
    \int_{B(x_o, R)} \norma{\nabla u_h(x)} \, \mathrm{d}x
    \leq
    \frac{\tilde C}{C_1} \,.
  \end{equation}
  Finally when $R \to \infty$ we get $\int_{\reali^N} \norma{\nabla
    u_h(x)} \, \mathrm{d}x \leq \tilde C / C_1$ and in the limit $h\to
  0$, by Theorem~\ref{thm:AFP} also $\tv(u) \leq \tilde C / C_1$,
  concluding the proof of the first statement.

  Assume now that $u \in \C1(\reali^N;\reali)$. Then, using the same
  computations as above,
  \begin{eqnarray*}
    & &
    \lim_{\gd \to 0}
    \frac{1}{\gd}
    \int_{\reali^N} \int_{\reali^N}
    \modulo{u (x)-u(x-z)} \, \mu(z)
    \, \mathrm{d}x \, \mathrm{d}z
    \\
    & = &
    \lim_{\gd \to 0}
    \int_{\reali^N} \int_{\reali^N}
    \modulo{\int_0^1 \nabla u(x-\gd s z) \cdot z \,\mathrm{d}s}
    \, \mu_1 (\norma{z})
    \, \mathrm{d}x \, \mathrm{d}z
    \\
    & = &
    C_1\, \tv(u) \,,
  \end{eqnarray*}
  completing the proof.
\end{proof}

In the following proof, this property of any function $u \in
\BV(\reali^N;\reali)$ will be of use:
\begin{equation}
  \label{eq:afp}
  \int_{\reali^N} \modulo{u(x)-u(x-z)} \, \mathrm{d}x
  \leq
  \norma{z} \, \tv(u) 
  \qquad \mbox{ for all } z \in \reali^N .
\end{equation}
For a proof, see~\cite[Remark~3.25]{AmbrosioFuscoPallara}.

\begin{proofof}{Theorem~\ref{teo:tv}}
  Assume first that $u_o \in \C1(\reali^N;\reali)$, the general case
  will be considered only at the end of this proof.

  Let $u$ be the weak entropy solution to~(\ref{eq:probv}). Denote $u
  = u(t,x)$ and $v = u(s,y)$ for $(t,x), (s,y) \in \rpis \times
  \reali^N$. Then, for all $k,l \in \reali$ and for all test functions
  $\phi = \phi(t,x,s,y)$ in $\Cc1 \left((\rpis \times \reali^N)^2;
    \rpic \right)$, we have
  \begin{equation}
    \label{eq:utv}
    \!\!
    \begin{array}{r}
      \displaystyle
      \!\!\int_{\rpis} \!\! \int_{\reali^N} \!\!
      \left[
        (u-k) \, \dt \phi
        +
        \left( f(t,x,u) - f(t,x,k) \right) \nabla_x \phi
        +
        \left(F(t,x,u)-\div f(t,x,k)\right)\phi
      \right]
      \\
      \times
      \mathrm{sign}(u-k) \,
      \mathrm{d}x \, \mathrm{d}t
      \geq
      0
    \end{array}
  \end{equation}
  for all $(s,y) \in \rpis \times \reali^N$, and
  \begin{equation}
    \label{eq:vtv}
    \!\!
    \begin{array}{r}
      \displaystyle
      \!\!\!\!\!\!
      \int_{\rpis} \!\! \int_{\reali^N} \!\!
      \left[
        (v-l) \, \ds \phi
        +
        \left( f(s,y,v) - f(s,y,l) \right) \nabla_y \phi
        +
        (F(s,y,v)-\div f(s,y,l))\phi
      \right]
      \\
      \times
      \mathrm{sign}(v-l) \,
      \mathrm{d}y \, \mathrm{d}s
      \geq
      0
    \end{array}
  \end{equation}
  for all $(t,x) \in \rpis \times \reali^N$. Let $\Phi \in \Cc\infty
  (\rpis \times \reali^N;\rpic)$, $\Psi \in \Cc\infty (\reali \times
  \reali^N;\rpic)$ and set
  \begin{equation}
    \label{eq:phi}
    \varphi(t,x,s,y)=\Phi(t,x) \, \Psi(t-s,x-y) \,.
  \end{equation}
  Observe that $\dt \varphi + \partial_s \phi = \Psi \, \dt \Phi$,
  $\nabla_x \phi = \Psi \, \nabla_x\Phi + \Phi \, \nabla_x\Psi$,
  $\nabla_y \phi = -\Phi \, \nabla_x\Psi$. Choose $k = v(s,y)$
  in~(\ref{eq:utv}) and integrate with respect to
  $(s,y)$. Analogously, take $l = u(t,x)$ in~(\ref{eq:vtv}) and
  integrate with respect to $(t,x)$. Summing the obtained
  inequalities, we get
  \begin{equation}
    \label{eq:sumtv}
    \!\!
    \begin{array}{r}
      \!\!  \displaystyle
      \int_{\rpis} \!\! \int_{\reali^N} \!\!
      \int_{\rpis} \!\! \int_{\reali^N} \!\!
      \!\!\!\!
      \mathrm{sign}(u-v)
      \bigg[
      (u-v) \, \Psi \, \dt \Phi +
      \left( f(t,x,u) - f(t,x,v) \right) \cdot
      \left( \nabla \Phi \right) \Psi
      \\
      +
      \left( f(s,y,v) - f(s,y,u) - f(t,x,v) + f(t,x,u) \right) \cdot
      \left( \nabla \Psi \right) \Phi
      \\
      +
      \left(F(t,x,u) - F(s,y,v) + \div f(s,y,u) - \div f(t,x,v) \right) \phi
      \bigg]
      \mathrm{d}x \, \mathrm{d}t \, \mathrm{d}y \, \mathrm{d}s
      \geq
      0 .
    \end{array}
  \end{equation}
  Introduce a family of functions $\{Y_\vartheta\}_{\vartheta>0}$ such
  that for any $\vartheta>0$:\\
  \begin{minipage}{0.5\linewidth}
    \centering
    \begin{psfrags}
      \psfrag{YT}{$Y_\vartheta$} \psfrag{1}{$1$} \psfrag{0}{$0$}
      \psfrag{e}{$\vartheta$} \psfrag{t}{$t$}
      \includegraphics[width=3.5cm]{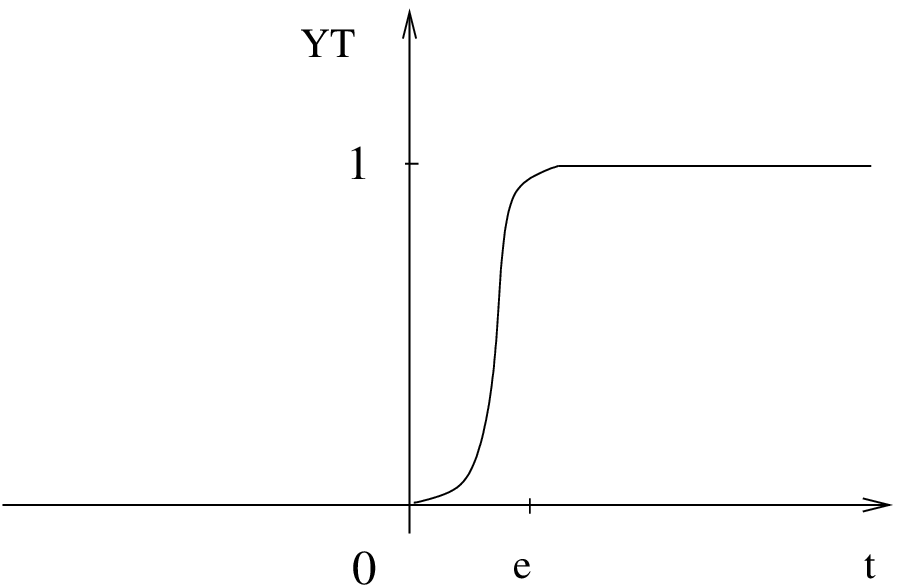}
    \end{psfrags}

    \begin{psfrags}
      \psfrag{Yp}{$Y'_\vartheta$} \psfrag{1}{$1$} \psfrag{0}{$0$}
      \psfrag{e}{$\vartheta$} \psfrag{t}{$t$}
      \includegraphics[width=3.5cm]{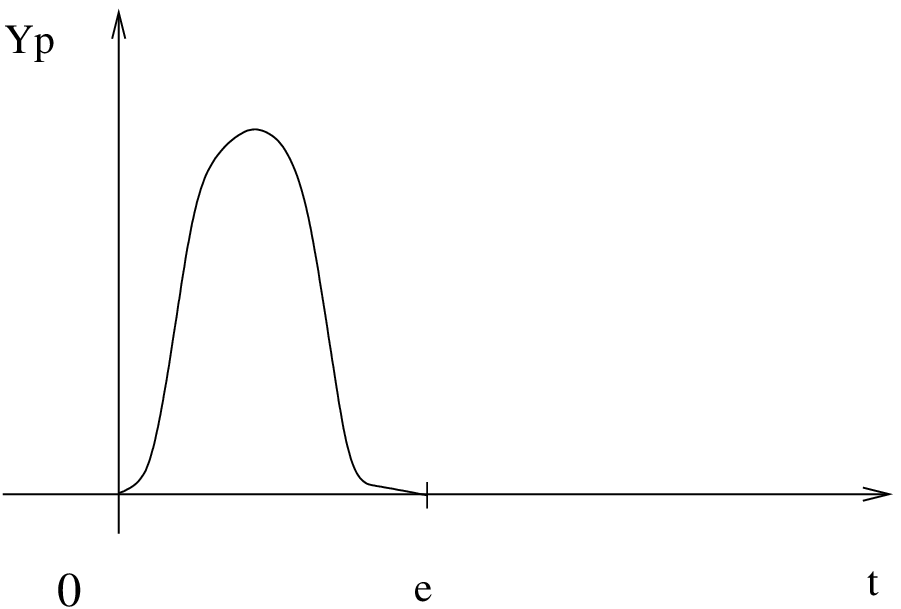}
    \end{psfrags}
  \end{minipage}\hfil
  \begin{minipage}{0.45\linewidth}
    \begin{equation}
      \label{eq:Y}
      \begin{array}{rcl}
        Y_\vartheta(t)
        & = &
        \displaystyle
        \int_{-\infty}^t Y_\vartheta' (s) \,\mathrm{d}s
        \\
        Y_\vartheta'(t)
        & = &
        \displaystyle
        \frac{1}{\vartheta} \, Y' \left( \frac{t}{\vartheta} \right)
        \\
        Y' & \in & \Cc\infty(\reali;\reali)
        \\
        \mathrm{supp}(Y')
        & \subset &
        \left]0,1\right[
        \\
        Y'
        &\geq &
        0
        \\
        \displaystyle
        \int_\reali Y'(s) \, \mathrm{d}s
        & = &
        1 \,.
      \end{array}
    \end{equation}
  \end{minipage}

  \noindent Let $M = \norma{\pt_u f}_{\L\infty(\Omega; \reali^{N})}$
  and define for $\epsilon, \theta, T_o, R > 0$, $x_o \in \reali^N$,
  \begin{figure}[htbp]
    \centering
    \begin{psfrags}
      \psfrag{chi}{$\chi$} \psfrag{1}{$1$} \psfrag{0}{$0$}
      \psfrag{t}{$t$} \psfrag{Te}{$T+\epsilon$} \psfrag{e}{$\epsilon$}
      \psfrag{T}{$T$}
      \includegraphics[width=0.45\linewidth]{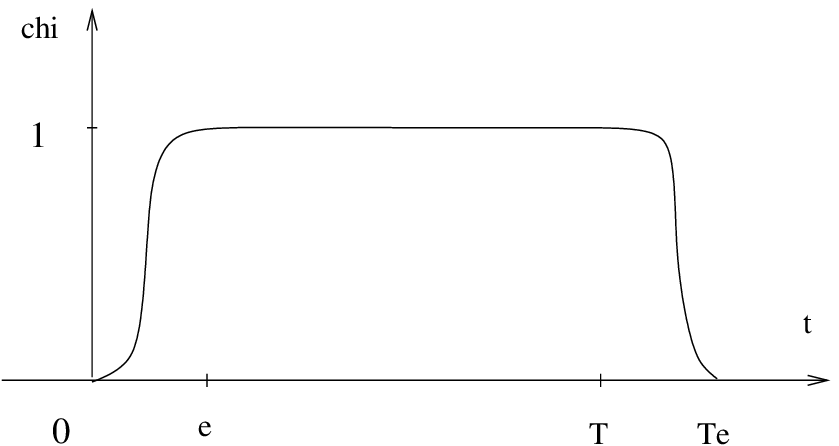}
    \end{psfrags}
    \begin{psfrags}
      \psfrag{psi}{$\psi$} \psfrag{1}{$1$} \psfrag{xo}{$x_o$}
      \psfrag{x}{$x$} \psfrag{a}{$a$} \psfrag{b}{$b$}
      \includegraphics[width=0.45\linewidth]{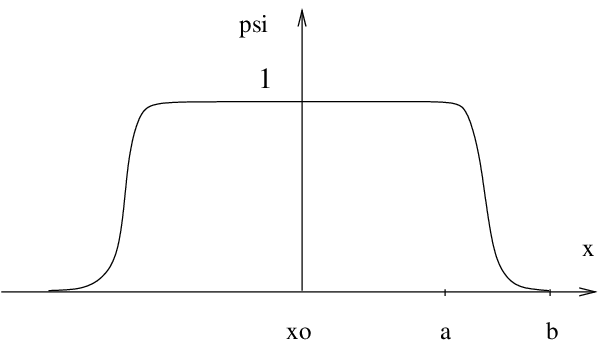}
    \end{psfrags}
    \caption{Graphs of $\chi$, left, and of $\psi$, right. Here
      $a=R+M(T_o-t)$ and $b=R+M(T_o-t)+\theta$.}
    \label{fig:chipsi}
  \end{figure}
  (see Figure~\ref{fig:chipsi}):
  \begin{equation}
    \label{eq:chipsi}
    \chi(t) = Y_\epsilon(t)-Y_\epsilon(t-T)
    \quad \mbox{ and } \quad
    \psi(t,x) = 1-Y_\theta\left(\norma{x-x_o}-R-M(T_o-t)\right)
    \geq 0,
  \end{equation}
  where we also need the compatibility conditions $T_o \geq T$ and $M
  \epsilon \leq R + M (T_o - T)$.  Observe that $\chi \to
  \mathbf{1}_{[0,T]}$ and $\chi' \to \delta_0 - \delta_T$ as
  $\epsilon$ tends to $0$.  On $\chi$ and $\psi$ we use the bounds
  \begin{displaymath}
    \chi
    \leq
    \mathbf{1}_{[0,T+\epsilon]}
    \quad \mbox{ and } \quad
    \mathbf{1}_{B(x_o, R+M(T_o-t))}
    \leq
    \psi
    \leq
    \mathbf{1}_{B(x_o, R+M(T_o-t)+\theta)} \,.
  \end{displaymath}
  In~(\ref{eq:sumtv}), choose $\Phi(t,x) = \chi(t) \, \psi(t,x)$. With
  this choice, we have
  \begin{equation}
    \label{eq:Derivatives}
    \dt \Phi 
    = 
    \chi' \, \psi - M \, \chi \,  Y_\theta'
    \quad \mbox{ and } \quad
    \nabla \Phi 
    = -\chi \, Y_\theta' \, \frac{x-x_o}{\norma{x-x_o}} \,.
  \end{equation}
  Setting $\displaystyle B(t,x,u,v) = \modulo{u-v} M +
  \mathrm{sign}(u-v) \left( f(t,x,u) - f(t,x,v) \right) \cdot
  \frac{x-x_o}{\norma{x-x_o}}$, the first line in~(\ref{eq:sumtv})
  becomes
  \begin{eqnarray*}
    \!\!
    & &
    \int_{\rpis} \! \int_{\reali^N} \! \int_{\rpis} \! \int_{\reali^N} \!\!
    \left[
      (u-v) \Psi \, \dt \Phi +
      \! \left( f(t,x,u) - f(t,x,v) \right)
      ( \nabla \Phi) \, \Psi
    \right]\!
    \mathrm{sign}(u-v)
    \mathrm{d}x \, \mathrm{d}t \, \mathrm{d}y \, \mathrm{d}s
    \\
    \!\!
    & = &
    \int_{\rpis} \int_{\reali^N}
    \int_{\rpis} \int_{\reali^N}
    \left(
      \modulo{u-v}\, \chi' \, \psi - B(t,x,u,v) \chi \, Y_\theta'
    \right)
    \, \Psi
    \, \mathrm{d}x \, \mathrm{d}t \, \mathrm{d}y \, \mathrm{d}s
    \\
    \!\!
    & \leq &
    \int_{\rpis} \int_{\reali^N}
    \int_{\rpis} \int_{\reali^N}
    \modulo{u-v} \, \chi' \, \psi \, \Psi
    \, \mathrm{d}x \, \mathrm{d}t \, \mathrm{d}y \, \mathrm{d}s
  \end{eqnarray*}
  since $B(t,x,u,v)$ is positive for all $(t, x, u, v) \in \Omega
  \times \reali$. Thanks to the above estimate and
  to~(\ref{eq:sumtv}), we have
  \begin{displaymath}
    \begin{array}{llcr}
      \displaystyle
      \!\!\!\!\int_{\rpis} \! \int_{\reali^N} \! \int_{\rpis} \! \int_{\reali^N} \!\!\!\!
      &
      \bigg[
      (u-v) \, \chi' \, \psi \, \Psi
      \\
      &
      \displaystyle
      +
      \left(
        f(s,y,v) - f(s,y,u) - f(t,x,v) + f(t,x,u)
      \right)
      \cdot
      (\nabla \Psi) \, \Phi
      \\
      &
      \displaystyle
      +
      \left(
        F(t,x,u) - F(s,y,v) - \div  f(t,x,v) + \div f(s,y,u)
      \right)
      \phi
      \bigg]
      \\
      &
      \displaystyle
      \times \mathrm{sign}(u-v)
      \, \mathrm{d}x\,\mathrm{d}t \, \mathrm{d}y \, \mathrm{d}s
      & \geq &
      0 .
    \end{array}
  \end{displaymath}
  Now, we aim at bounds for each term of this sum. Introduce the
  following notations:
  \begin{eqnarray*}
    I
    & = &
    \int_{\rpis} \int_{\reali^N} \int_{\rpis} \int_{\reali^N}
    \modulo{u-v} \, \chi' \, \psi \, \Psi
    \, \mathrm{d}x\,\mathrm{d}t \, \mathrm{d}y \, \mathrm{d}s\, ,
    \\
    J_x
    & = &
    \int_{\rpis} \int_{\reali^N} \int_{\rpis} \int_{\reali^N}
    \left( f(t,y,v) - f(t,y,u) + f(t,x,u) - f(t,x,v) \right)
    \left( \nabla \Psi \right) \, \Phi
    \\ & &
    \qquad\qquad\qquad\qquad
    \times \,
    \mathrm{sign}(u-v)
    \, \mathrm{d}x\,\mathrm{d}t \, \mathrm{d}y \, \mathrm{d}s\, ,
    \\
    J_t
    & = &
    \int_{\rpis} \int_{\reali^N} \int_{\rpis} \int_{\reali^N}
    \left( f(s,y,v) - f(s,y,u) + f(t,y,u) - f(t,y,v) \right)
    \left( \nabla \Psi \right) \, \Phi
    \\
    & &
    \qquad\qquad\qquad\qquad
    \times \,
    \mathrm{sign}(u-v)
    \, \mathrm{d}x\,\mathrm{d}t \, \mathrm{d}y \, \mathrm{d}s\, ,
    \\
    L_x
    & = &
    \int_{\rpis} \int_{\reali^N} \int_{\rpis} \int_{\reali^N}
    \left(
      F(t,x,u) - F(t,y,v) - \div  f(t,x,v) + \div f(t,y,u)
    \right) \, \varphi
    \\
    & &
    \qquad\qquad\qquad\qquad
    \times \,
    \mathrm{sign}(u-v)
    \, \mathrm{d}x\,\mathrm{d}t \, \mathrm{d}y \, \mathrm{d}s\, ,
    \\
    L_t
    & = &
    \int_{\rpis} \int_{\reali^N} \int_{\rpis} \int_{\reali^N}
    \left(
      F(t,y,v) - F(s,y,v) - \div  f(t,y,u) + \div f(s,y,u)
    \right) \, \varphi
    \\
    & &
    \qquad\qquad\qquad\qquad
    \times \,
    \mathrm{sign}(u-v)
    \, \mathrm{d}x\,\mathrm{d}t \, \mathrm{d}y \, \mathrm{d}s \,.
  \end{eqnarray*}
  Then, the above inequality is rewritten as $I + J_x+J_{t} +
  L_x+L_{t} \geq 0$.  Choose $\Psi(t,x) = \nu(t)\, \mu(x)$ where, for
  $\pd, \gd > 0$, $\mu \in \Cc\infty (\rpic; \rpic)$
  satisfies~(\ref{eq:mu})--(\ref{eq:Mu}) and
  \begin{equation}
    \label{eq:nu}
    \nu(t)
    =
    \frac{1}{\pd} \, \nu_1 \left( \frac{t}{\pd} \right)
    \,,\quad
    \int_\reali \nu_1(s) \,\mathrm{d} s
    =
    1
    \,,\quad
    \nu_1
    \in
    \Cc\infty(\reali;\rpic)
    \,,\quad
    \mathrm{supp}(\nu_1)
    \subset
    \left]-1,0\right[
    \,.
  \end{equation}
  We have
  \begin{eqnarray*}
    I
    & \leq &
    I_1 + I_ 2 \qquad \mbox{ where}
    \\
    I_1
    & = &
    \int_{\rpis} \int_{\reali^N} \int_{\rpis} \int_{\reali^N}
    \modulo{u(t,x)-u(t,y)}
    \, \left( Y_\epsilon'(t) - Y_\epsilon'(t-T) \right) \, \psi \, \Psi
    \, \mathrm{d}x\,\mathrm{d}t \, \mathrm{d}y \, \mathrm{d}s\, ,
    \\
    I_2
    & = &
    \int_{\rpis} \int_{\reali^N} \int_{\rpis} \int_{\reali^N}
    \modulo{u(t,y)-u(s,y)}
    \, \left(Y_\epsilon'(t) + Y_\epsilon'(t-T) \right) \, \psi
    \, \Psi
    \, \mathrm{d}x\,\mathrm{d}t \, \mathrm{d}y \, \mathrm{d}s
  \end{eqnarray*}
  and we get
  \begin{eqnarray*}
    \limsup_{\epsilon\to 0} I_1
    & \leq &
    \int_{\reali^N} \int_{\norma{x-x_o}\leq R+MT_o+\theta }
    \modulo{u(0,x)-u(0,y)}
    \, \mu(x-y) \, \mathrm{d}x \, \mathrm{d}y
    \\
    & &
    -
    \int_{\reali^N} \int_{\norma{x-x_o}\leq R +M(T_o-T)}
    \modulo{u(T,x)-u(T,y)}
    \, \mu(x-y) \, \mathrm{d}x \, \mathrm{d}y\, ,
    \\
    \limsup_{\epsilon\to 0} I_2
    & \leq &
    2\, \sup_{t\in \{0,T\}, \atop s \in \left]t,t+\pd\right[ }
    \int_{\norma{y-x_o}\leq R+\gd+M(T_o-t)+\theta}
    \modulo{u(t,y)-u(s,y)} \, \mathrm{d}y \,.
  \end{eqnarray*}
  For $J_x$, we have that by~\textbf{(H1)}, $f \in \C2( \Omega ;
  \reali^N)$ and therefore
  \begin{eqnarray*}
    & &
    \norma{f(t,y,v)-f(t,y,u)+f(t,x,u)-f(t,x,v)} =
    \\
    & = &
    \norma{
      \int_{u}^{v} \int_0^1
      \nabla\pt_u   f\left(t,x(1-r)+ry, w\right) \cdot (y-x)
      \,\mathrm{d}r \, \mathrm{d}w
    }
    \\
    & \leq &
    \norma{\nabla\pt_u f }_{\L\infty(\Omega;\reali^{N\times
        N})}
    \norma{x-y} \, \modulo{u(s,y)-u(t,x)} \,.
  \end{eqnarray*}
  Then, using~(\ref{eq:mu3})
  \begin{eqnarray*}
    J_x
    & \leq &
    \norma{\nabla\pt_u f }_{\L\infty}
    \int_{\rpis} \int_{\reali^N} \int_{\rpis} \int_{\reali^N}
    \norma{x-y} \, \modulo{u(t,x)-u(s,y)}
    \norma{\nabla \Psi} \, \chi \, \psi
    \,\mathrm{d}x\,\mathrm{d}t\,\mathrm{d}y\,\mathrm{d}s
    \\
    & \leq &
    \norma{\nabla\pt_u f }_{\L\infty}
    \int_{\rpis} \int_{\reali^N} \int_{\rpis} \int_{\reali^N}
    \norma{x-y} 
    \left[\modulo{u(t,y)-u(s,y)}+\modulo{u(t,x)-u(t,y)}\right]
    \\
    & &
    \qquad \qquad \qquad \qquad \qquad \qquad
    \times
    \norma{\nabla \Psi} \, \chi \, \psi
    \,\mathrm{d}x\,\mathrm{d}t\,\mathrm{d}y\,\mathrm{d}s
    \\
    & \leq &
    N \,
    \norma{\nabla\pt_u f }_{\L\infty} (T+\epsilon)
    \sup_{t\in [0,T+\epsilon], \atop s \in \left]t, t+\pd\right[}
    \int_{\norma{y-x_o}\le R+\gd+M(T_o-t)+\theta }
    \modulo{u(t,y)-u(s,y)} \, \mathrm{d}y
    \\
    & &
    +
    \norma{\nabla\pt_u f }_{\L\infty}
    \int_0^{T+\epsilon}
    \int_{\reali^N}
    \int_{B(x_o, R +M(T_o-t)+\theta)} 
    \norma{x-y} \, \modulo{u(t,x)-u(t,y)} \, \norma{\nabla \mu}
    \,\mathrm{d}x\,\mathrm{d}y\,\mathrm{d}t\, ,
    \\
    J_{t}
    & \leq &
    \int_{\rpis} \int_{\reali^N} \int_{\rpis} \int_{\reali^N}
    \norma{
      \int_s^t \int_v^u \pt_{t} \pt_{u} f(\tau, y,w)
      \, \mathrm{d}w \, \mathrm{d}\tau
    }
    \norma{\nabla \Psi}\, \Phi
    \,\mathrm{d}x\,\mathrm{d}t\,\mathrm{d}y\,\mathrm{d}s
    \\
    & \leq &
    \pd \, \norma{\pt_{t} \pt_{u} f}_{\L\infty}
    \int_{\rpis} \int_{\reali^N} \int_{\rpis} \int_{\reali^N}
    \modulo{u(t,x)-u(s,y)}
    \, \norma{\nabla \Psi} \,\Phi
    \, \mathrm{d}x \, \mathrm{d}t \, \mathrm{d}y \, \mathrm{d}s \,.
  \end{eqnarray*}
  For $L_x$, we get
  \begin{eqnarray*}
    L_x
    & = &
    L_1 + L_2 \qquad \mbox{ where}
    \\
    L_1
    & = &
    \int_{\rpis} \int_{\reali^N} \! \int_{\rpis} \int_{\reali^N}
    \!
    \left[
      \int_v^u \left(\pt_u \div  f(t,x,w)+\pt_u F(t,y,w) \right) \mathrm{d}w
    \right]\!
    \phi \, \mathrm{sign}(u-v)
    \,\mathrm{d}x\,\mathrm{d}t\,\mathrm{d}y\,\mathrm{d}s ,
    \\
    L_2
    & = &
    \int_{\rpis} \int_{\reali^N} \! \int_{\rpis} \int_{\reali^N}
    \left[
      \int_0^1 \! \nabla ( F - \div f) \left(t,rx+(1-r)y,u\right)\cdot (x-y)
      \, \mathrm{d}r
    \right]
    \phi
    \\
    & &
    \qquad \qquad \qquad \qquad
    \times \mathrm{sign}(u-v)
    \,\mathrm{d}x\,\mathrm{d}t\,\mathrm{d}y\,\mathrm{d}s.
  \end{eqnarray*}
  Then, recalling~(\ref{eq:phi}), the definitions $\Psi = \nu \,\mu$,
  $\Phi = \chi \, \psi$, (\ref{eq:mu}), (\ref{eq:nu})
  and~(\ref{eq:chipsi}), we obtain
  \begin{eqnarray*}
    L_1
    &\leq &
    \left(
      N \norma{\nabla\pt_u f}_{\L\infty}
      +
      \norma{\pt_u F}_{\L\infty}
    \right)
    \\
    & &
    \times
    \bigg[
    (T+\epsilon)
    \sup_{ t\in [0,T+\epsilon], \atop s \in \left]t, t+\pd\right[}
    \int_{\norma{y-x_o}\leq R+\gd+M(T_o-t)+\theta}
    \modulo{u(t,y)-u(s,y)} \, \mathrm{d}y
    \\
    & &
    \qquad
    +
    \int_0^{T+\epsilon} \int_{\reali^N}
    \int_{\norma{x-x_o}\leq R+M(T_o-t)+\theta}
    \modulo{u(t,x)-u(t,y)} \, \mu(x-y)
    \,\mathrm{d}x\,\mathrm{d}y\,\mathrm{d}t
    \bigg]\, ,
    \\
    L_2
    & \leq &
    \int_{\rpis} \! \int_{\reali^N} \! \int_{\rpis} \! \int_{\reali^N} \!\!
    \int_0^1
    \norma{\nabla( F - \div f) \left(t,y+r(x-y),u\right)}
    \norma{x-y} 
    \chi \, \psi \, \mu \, \nu
    \mathrm{d}r
    \,\mathrm{d}x\,\mathrm{d}t\,\mathrm{d}y\,\mathrm{d}s
    \\
    & \leq &
    \left(
      \int_0^{T+\varepsilon} \int_{\reali^N}
      \norma{\nabla (F - \div f) (t, y,\cdot)}_{\L\infty}
      \mathrm{d}y\, \mathrm{d}t
    \right)
    \int_{\reali^N} \norma{x} \, \mu(x) \, \mathrm{d}x
    \\
    & = &
    \gd \, M_1
    \;
    \int_0^{T+\varepsilon} \int_{\reali^N}
    \norma{\nabla (F - \div f) (t, y,\cdot)}_{\L\infty}\, \mathrm{d}y\, \mathrm{d}t
  \end{eqnarray*}
  where
  \begin{equation}
    \label{eq:M1}
    M_1
    =
    \int_{\reali^N} \norma{x} \,
    \mu_1\left(\norma{x}\right) \, \mathrm{d}x \,.
  \end{equation}
  Concerning the latter term $L_t$
  \begin{eqnarray*}
    L_{t}
    & \leq &
    \pd \, \omega_N \, (R+M T_o)^N \, (T+\epsilon)
    \left(
      \norma{\pt_t \div f}_{\L\infty}
      +
      \norma{\pt_t F}_{\L\infty}
    \right) \,.
  \end{eqnarray*}
  Letting $\epsilon, \pd, \theta \to 0$ we get
  \begin{eqnarray*}
    \limsup_{\epsilon, \pd, \theta \to 0} I_1
    & = &
    \int_{\reali^N} \int_{\norma{x-x_o}\leq R+MT_o}
    \modulo{u(0,x)-u(0,y)}
    \, \mu(x-y) \, \mathrm{d}x \, \mathrm{d}y
    \\
    & &
    -
    \int_{\reali^N} \int_{\norma{x-x_o}\leq R +M(T_o-T)}
    \modulo{u(T,x)-u(T,y)}
    \, \mu(x-y) \, \mathrm{d}x \, \mathrm{d}y\, ,
    \\
    \limsup_{\epsilon, \pd, \theta \to 0} I_2
    & = &
    0\, ,
    \\
    \limsup_{\epsilon, \pd, \theta \to 0} J_x
    & \leq &
    \norma{\nabla\pt_u f }_{\L\infty}
    \int_0^{T}
    \int_{\reali^N}
    \int_{B(x_o, R +M(T_o-t))}
    \norma{x-y} \, \modulo{u(t,x)-u(t,y)}
    \\
    & &
    \qquad \qquad
    \times
    \norma{\nabla \mu(x - y)}
    \,\mathrm{d}x\,\mathrm{d}y\,\mathrm{d}t\, ,
    \\
    \limsup_{\epsilon, \pd, \theta \to 0} J_{t}
    & = &0\, ,
    \\
    \limsup_{\epsilon, \pd, \theta \to 0} L_1
    & \leq &
    \left(
      N \norma{\nabla\pt_u f}_{\L\infty}
      +
      \norma{\pt_u F}_{\L\infty}
    \right)
    \\
    & &
    \qquad
    \times
    \int_0^{T} \int_{\reali^N}
    \int_{\norma{x-x_o}\leq R+M(T_o-t)}
    \modulo{u(t,x)-u(t,y)} \, \mu(x-y)
    \,\mathrm{d}x\,\mathrm{d}y\,\mathrm{d}t\, ,
    \\
    \limsup_{\epsilon, \pd, \theta \to 0} L_2
    & \leq &
    \gd  M_1
    \int_0^{T} \int_{\reali^N}
    \norma{\nabla(F - \div f) (t,y,\cdot)}_{\L\infty}
    \, \mathrm{d}y\,\mathrm{d}t\,  ,
    \\
    \limsup_{\epsilon, \pd, \theta \to 0} L_{t}
    & = &
    0 \,.
  \end{eqnarray*}
  Collating all the obtained results and using the equality
  $\norma{\nabla \mu (x)} = - \frac{1}{ \gd^{N+1}} \mu_1'\left(
    \frac{\norma{x}}{\gd} \right)$
  \begin{equation}
    \label{inegalite}
    \begin{array}{rcl}
      & &
      \displaystyle
      \int_{\reali^N} \int_{\norma{x-x_o}\leq R +M(T_o-T)}
      \modulo{u(T,x)-u(T,y)}
      \, \frac{1}{\gd^N}
      \mu_1\left(\frac{\norma{x-y}}{\gd}\right)
      \, \mathrm{d}x \, \mathrm{d}y
      \\[10pt]
      & \leq &
      \displaystyle
      \int_{\reali^N} \int_{\norma{x-x_o}\leq R +M(T_o-T)}
      \modulo{u(0,x)-u(0,y)}
      \, \frac{1}{\gd^N}
      \mu_1\left(\frac{\norma{x-y}}{\gd}\right)
      \, \mathrm{d}x \, \mathrm{d}y
      \\[10pt]
      & &
      \displaystyle
      -
      \norma{\nabla\pt_u f }_{\L\infty}
      \int_0^{T} \int_{\reali^N} \int_{\norma{x-x_o}\leq R +M(T_o-t)} \!\!\!
      \modulo{u(t,x)-u(t,y)}
      \\[10pt]
      & &
      \displaystyle
      \qquad \qquad \qquad
      \times
      \frac{1}{\gd^{N+1}} \, \mu_1'\left( \frac{\norma{x-y}}{\gd} \right)
      \, \norma{x-y}
      \,\mathrm{d}x\,\mathrm{d}y\,\mathrm{d}t
      \\[10pt]
      & &
      \displaystyle
      +
      \left(
        N \norma{\nabla\pt_u f}_{\L\infty}
        +
        \norma{\pt_u F}_{\L\infty}
      \right) \!
      \int_0^{T} \! \int_{\reali^N} \! \int_{\norma{x-x_o}\leq R
        +M(T_o-t)} \!
      \modulo{u(t,x)-u(t,y)}
      \\[10pt]
      & &
      \displaystyle
      \qquad \qquad \qquad
      \times
      \frac{1}{\gd^N}
      \mu_1\left(\frac{\norma{x-y}}{\gd}\right)
      \,\mathrm{d}x\,\mathrm{d}y\,\mathrm{d}t
      \\[10pt]
      & &
      \displaystyle
      +
      \gd  M_1
      \int_0^{T}\int_{\reali^N} 
      \norma{\nabla(F - \div f)(t, y,\cdot)}_{\L\infty} \,
      \mathrm{d}y\, \mathrm{d}t \,.
    \end{array}
  \end{equation}

  If $\norma{\nabla\pt_u f }_{\L\infty} =\norma{\pt_u F}_{\L\infty}=
  0$ and under the present assumption that $u_o \in \C1
  (\reali^N;\reali)$, using Proposition~\ref{prop:normtv},
  (\ref{eq:C1}) and~(\ref{eq:M1}), we directly obtain that
  \begin{equation}
    \label{eq:CK}
    \tv(u(T))
    \leq
    \tv(u_o)
    +
    \frac{M_1}{C_1}  \,
    \int_0^{T} \int_{\reali^N} 
    \norma{\nabla (F - \div f)(t, y,\cdot)}_{\L\infty}
    \mathrm{d}y\, \mathrm{d}t \,.
  \end{equation}
  The same procedure at the end of this proof allows to
  extend~(\ref{eq:CK}) to more general initial data, providing an
  estimate of $\tv \left( u(t) \right)$ in the situation studied
  in~\cite{bouchutperthame}.

  Now, it remains to treat the case $\norma{\nabla \pt_u f}_{\L\infty}
  \neq 0$. A direct use of Gronwall type inequalities is apparently
  impossible, due to the term with $\nabla \mu$.  However, introduce
  the function
  \begin{displaymath}
    \mathcal{F}(T,\gd)
    =
    \int_0^T \int_{\reali^N} \int_{\norma{x-x_o} \leq R+M(T_o-t)}
    \modulo{u(t,x)-u(t,x-z)}
    \, \frac{1}{\gd^N}
    \, \mu_1\left(\frac{\norma{z}}{\gd}\right)
    \, \mathrm{d}x \, \mathrm{d}z \, \mathrm{d}t
  \end{displaymath}
  so that
  \begin{eqnarray*}
    \partial_\gd \mathcal{F}
    & = &
    - \frac{N}{\gd} \mathcal{F}
    \\
    & &
    - \frac{1}{\gd}
    \int_0^{T} \int_{\reali^N} \int_{\norma{x-x_o}\le R + M(T_o-t)}
    \modulo{u(t,x)-u(t,x-z)}
    \, \frac{\mu_1'\left( \norma{z} /\gd \right)}{\gd^{N+1}} \,
    \, \norma{z}
    \, \mathrm{d}x \, \mathrm{d}z \, \mathrm{d}t \,.
  \end{eqnarray*}
  Denote $\displaystyle C(T) = M_1 \int_0^{T} \int_{\reali^N}
  \norma{\nabla ( F - \div f) (t, y,\cdot)}_{\L\infty} \mathrm{d}y\,
  \mathrm{d}t$ and integrate~(\ref{inegalite}) on $[0, T']$ with
  respect to $T$ for $T' \leq T_o$. It results
  \begin{eqnarray*}
    \frac{1}{\gd}\mathcal{F}(T',\gd)
    & \leq &
    \frac{T'}{\gd} \int_{\reali^N} \int_{\norma{x-x_o}\leq R + MT_o}
    \modulo{u(0,x)-u(0,y)} \, \mu(x-y) \, \mathrm{d}x \, \mathrm{d}y
    \\
    & &
    +
    T' \, \norma{\nabla\pt_u f }_{\L\infty} \, \partial_\gd \mathcal{F}(T',\gd)
    +
    \frac{T'}{\gd} \,
    \left(
      2N \norma{\nabla\pt_u f}_{\L\infty} +\norma{\pt_u F}_{\L\infty}
    \right)
    \, \mathcal{F}(T',\gd)
    \\
    & &
    +
    T' \, C(T') \,.
  \end{eqnarray*}
  Denote $\alpha = \left( 2 N \norma{\nabla\pt_u
      f}_{\L\infty}+\norma{\pt_u F}_{\L\infty} - \frac{1}{T'} \right)
  \left(\norma{\nabla\pt_u f}_{\L\infty} \right)^{-1}$, so that
  $\lim_{T' \to 0} \alpha = -\infty$. The previous inequality reads,
  using~(\ref{eq:afp}) for $u_o$,
  \begin{eqnarray*}
    \partial_\gd \mathcal{F}(T',\gd)
    +
    \alpha \,
    \frac{\mathcal{F}(T',\gd)}{\gd}
    & \geq &
    - \left( M_1\tv(u_o) +  C(T') \right)
    \frac{1}{\norma{\nabla\pt_u f}_{\L\infty}} \,,
    \\
    \partial_\gd
    \left( \gd^\alpha \, \mathcal{F}(T',\gd) \right)
    & \geq &
    - \gd^\alpha \left( M_1\tv(u_o) +  C(T') \right)
    \frac{1}{\norma{\nabla\pt_u f }_{\L\infty}} \,.
  \end{eqnarray*}
  Finally, if $T'$ is such that $\alpha < -1$, then we integrate in
  $\gd$ on $\left[ \gd, +\infty \right[$ and we get
  \begin{equation}
    \label{eq:FDelta}
    \frac{1}{\gd} \mathcal{F}(T',\gd)
    \leq
    \frac{1}{-\alpha-1} \left( M_1\tv{(u_o)} +  C(T') \right)
    \frac{1}{\norma{\nabla\pt_u f}_{\L\infty}} \,.
  \end{equation}
  Furthermore, by~(\ref{eq:mu}) and~(\ref{eq:Mu}), there exists a
  constant $K > 0$ such that for all $z \in \reali^N$
  \begin{equation}
    \label{eq:BoundMu}
    -\mu_1'(\norma{z})
    \leq
    K \mu_1\left(\frac{\norma{z}}{2}\right) \,.
  \end{equation}
  Divide both sides in~(\ref{inegalite}) by $\gd$, rewrite them
  using~(\ref{eq:FDelta}), (\ref{eq:BoundMu}), apply~(\ref{eq:afp})
  and obtain
  \begin{eqnarray*}
    & &
    \frac{1}{\gd}
    \int_{\reali^N} \int_{\norma{x-x_o}\leq R +M(T_o-T)}
    \modulo{u(T,x) - u(T,y)}
    \, \frac{1}{\gd^N}
    \mu_1\left(\frac{\norma{x-y}}{\gd}\right)
    \, \mathrm{d}x \, \mathrm{d}y
    \\
    & \leq &
    M_1 \tv(u_o)
    +
    \frac{\mathcal{F}(T,2\gd)}{2\gd} \, 2^{N+2} \, K \,
    \norma{\nabla\pt_u f}_{\L\infty}
    +
    \frac{ \mathcal{F}(T,\gd)}{\gd}
    \left(
      2 N \norma{\nabla\pt_u f}_{\L\infty}+\norma{\pt_u F}_{\L\infty}
    \right)
    \\
    & &+
    M_1
    \int_0^{T} \int_{\reali^N} 
    \norma{\nabla(F - \div f)(t, y,\cdot)}_{\L\infty}
    \mathrm{d}y\, \mathrm{d}t\,.
  \end{eqnarray*}
  An application of~(\ref{eq:FDelta}) yields an estimate of the type
  \begin{equation}
    \label{ineq1}
    \frac{1}{\gd} \int_{\reali^N} \int_{B(x_o,R+M(T_o-T))}
    \modulo{u(T,x)-u(T,x-z)} \, \mu(z)
    \, \mathrm{d}x \, \mathrm{d}z
    \leq \check C \,,
  \end{equation}
  the positive constant $\check C$ being independent from $R$ and
  $\gd$. Applying Proposition~\ref{prop:normtv} we obtain that
  $u(t)\in \BV(\reali^N;\reali)$ for $t \in \left[ 0, 2T_1\right[$,
  where
  \begin{equation}
    \label{eq:T1}
    T_1
    =
    \frac{1}{2\left( (1+2N) \norma{\nabla\pt_u f}_{\L\infty}
        + \norma{\pt_u F}_{\L\infty}\right)} \,.
  \end{equation}

  \bigskip

  The next step is to obtain a general estimate of the $\tv$ norm. The
  starting point is~(\ref{inegalite}). Recall the
  definitions~(\ref{eq:M1}) of $M_1$ and~(\ref{eq:T1}) of
  $T_1$. Moreover, by~(\ref{eq:mu4}),
  \begin{displaymath}
    \int_{\reali^N} \norma{z}^2 \, \mu_1'(\norma{z}) \, \mathrm{d}z
    =
    - (N+1) \, M_1 \,.
  \end{displaymath}
  Divide both terms in~(\ref{inegalite}) by $\gd$,
  apply~(\ref{eq:TV1}) on the first term in the right hand side,
  apply~(\ref{eq:afp}) on the second and third terms and obtain for
  all $T \in [0,T_1]$ with $T_1 < T_o$
  \begin{eqnarray*}
    \tv \left( u(T) \right)
    & \leq &
    \tv(u_o)
    +
    \left(
      (2N+1) \norma{\nabla\pt_u f}_{\L\infty}
      +
      \norma{\pt_u F}_{\L\infty}
    \right)
    \frac{M_1}{C_1}  \int_0^T \tv \left( u(t) \right) \, \mathrm{d}t
    \\
    & &
    +
    \frac{M_1}{C_1}  \int_0^{T} \int_{\reali^N}
    \norma{\nabla(F - \div f)(t,x,\cdot)}_{\L\infty} 
    \mathrm{d}x \, \mathrm{d}t\,.
  \end{eqnarray*}
  An application of Gronwall Lemma shows that $\tv \left( u(t)
  \right)$ is bounded on $[0,T_1]$. Indeed,
  \begin{equation}
    \label{eq:TV}
    \tv \left( u(t) \right)
    \leq
    e^{\kappa_o t} \, \tv(u_o)
    +
    \frac{M_1}{C_1} \int_0^T e^{\kappa_o(T-t)}
    \int_{\reali^N} 
    \norma{\nabla(F-\div f)(t,x,\cdot)}_{\L\infty} 
    \, \mathrm{d}x\, \mathrm{d}t
  \end{equation}
  for $t \in [0,T_1]$, $M_1, C_1$ as in~(\ref{eq:M1}), (\ref{eq:C1})
  and $\kappa_o = [ (2N+1) \norma{\nabla\pt_u f}_{\L\infty} +
  \norma{\pt_u F}_{\L\infty} ] M_1/C_1$.

  We now relax the assumption on the regularity of $u_o$. Indeed, let
  $u_o \in \BV(\reali^N;\reali)$ and choose a sequence $u_o^n$ of
  $\C1(\reali^N;\reali)$ functions such that $\tv(u_o^n) \to
  \tv(u_o)$, as in Theorem~\ref{thm:AFP}. Then, by
  Theorem~\ref{teo:kruzkov}, the solutions $u^n$ to~(\ref{eq:probv})
  with initial datum $u^n_o$ satisfy
  \begin{displaymath}
    \lim_{n\to +\infty} u^n(t) = u(t) \mbox{ in } \Lloc1
    \quad \mbox{ and } \quad
    \tv \left( u(t) \right) 
    \leq 
    \liminf_{n\to +\infty} \tv \left( u^n(t) \right) ,
  \end{displaymath}
  where we used also the lower semicontinuity of the total
  variation. Note that~(\ref{eq:TV}), as well as the relations above,
  holds for all $t \in [0,T_1]$, $T_1$ being independent from the
  initial datum.  Therefore, the bound~(\ref{eq:TV}) holds for all
  $\BV$ initial data.

  Remark that the bound~(\ref{eq:TV}) is additive in time, in the
  sense that applying it iteratively for times $T_1$ and $t$
  yields~(\ref{eq:TV}) for time $T_1+t$:
  \begin{eqnarray*}
    & &
    \tv \left( u(T_1+t) \right) \!\!
    \\
    & \leq &
    e^{\kappa_o t} \, \tv\left(u(T_1)\right)
    +
    \frac{M_1}{C_1} \int_{T_1}^{T_1+t}e^{\kappa_o(t-s)}\int_{\reali^N} 
    \norma{\nabla(F-\div f)(s, x,\cdot)}_{\L\infty} \,
    \mathrm{d}x\, \mathrm{d}s
    \\
    & \leq &
    e^{\kappa_o t}
    \left[
      e^{\kappa_o T_1} \, \tv(u_o)
      +
      \frac{M_1}{C_1} \int_{0}^{T_1}e^{\kappa_o(T_1-s)} \int_{\reali^N} 
      \norma{\nabla(F - \div f)(s, x,\cdot)}_{\L\infty} \,
      \mathrm{d}x\, \mathrm{d}s
    \right]
    \\
    & &
    +
    \frac{M_1}{C_1} \int_{T_1}^{T_1+t}e^{\kappa_o(T_1+t-s)}\int_{\reali^N} 
    \norma{\nabla(F-\div f)(s, x,\cdot)}_{\L\infty} \,
    \mathrm{d}x\, \mathrm{d}s
    \\
    & = &
    e^{\kappa_o (T_1+t)} \, \tv(u_o)
    +
    \frac{M_1}{C_1}\int_{0}^{T_1+t}e^{\kappa_o(T_1+t-s)} \int_{\reali^N} 
    \norma{\nabla(F-\div f)(s, x,\cdot)}_{\L\infty} \,
    \mathrm{d}x\, \mathrm{d}s\, .
  \end{eqnarray*}
  The bound~(\ref{eq:TV}) can then be applied iteratively, thanks to
  the fact that $T_1$ is independent from the initial datum. An
  iteration argument allows to prove~(\ref{result}) for $t \in [0,
  T_o]$. The final bound~(\ref{result}) then follows by the
  arbitrariness of $T_o$, thanks to~(\ref{eq:W}).
\end{proofof}

\Section{Proof of
  Theorem~\ref{teo:estimates}.}\label{sec:proofcomparison}

The following proof relies on developing the techniques used in the
proof of Theorem~\ref{teo:tv}.

\begin{proofof}{Theorem~\ref{teo:estimates}}
  Let $\Phi\in \Cc\infty( \rpis \times \reali^N; \rpic)$, $\Psi \in
  \Cc\infty (\reali\times \reali^N; \rpic)$ and set
  $\varphi(t,x,s,y)=\Phi(t,x) \Psi(t-s,x-y)$ as in~(\ref{eq:phi}).

  By Definition~\ref{def:sol}, we have $\forall l\in \reali$, $\forall
  (t,x) \in \rpis \times \reali^N$
  \begin{equation}
    \label{eq:u}
    \begin{array}{r}
      \displaystyle
      \!\!\int_{\rpis} \!\! \int_{\reali^N} \!\!
      \left[
        (u-l) \, \pt_s \phi
        +
        \left( f(s,y,u) - f(s,y,l) \right) \cdot \nabla_y \phi
        +
        \left(F(s,y,u)-\div f(s,y,l)\right) \phi
      \right]
      \\
      \times \mathrm{sign}(u-l) \,
      \mathrm{d}y \, \mathrm{d}s
      \geq
      0
    \end{array}
  \end{equation}
  and $\forall k\in \reali$, $\forall (s,y)\in \rpis \times \reali^N$
  \begin{equation}
    \label{eq:v}
    \begin{array}{r}
      \displaystyle
      \!\!\int_{\rpis} \!\! \int_{\reali^N} \!\!
      \left[
        (v-k) \, \pt_t \phi
        +
        \left( g(t,x,v) - g(t,x,k) \right) \cdot \nabla_x \phi
        +
        \left(G(t,x,v) - \div g(t,x,k)\right)\phi
      \right]
      \\
      \times \mathrm{sign}(v-k) \,
      \mathrm{d}x \, \mathrm{d}t
      \geq
      0 .
    \end{array}
  \end{equation}
  Choose $k=u(s,y)$ in~(\ref{eq:v}) and integrate with respect to
  $(s,y)$. Analogously, take $l=v(t,x)$ in~(\ref{eq:u}) and integrate
  with respect to $(t,x)$. By summing the obtained equations, we get,
  denoting $u=u(s,y)$ and $v=v(t,x)$:
  \begin{equation}
    \label{eq:sum}
    \!\!\!
    \begin{array}{rl@{\;}c@{\;}l}
      \displaystyle 
      \int_{\rpis} \int_{\reali^N} \int_{\rpis} \int_{\reali^N} 
      &
      \displaystyle
      \bigg[
      (u-v) \Psi\dt \Phi 
      +
      \left(g(t,x,u)-g(t,x,v)\right) \cdot \left(\nabla\Phi\right) \Psi
      \\
      &
      +
      \left( g(t,x,u) - g(t,x,v) - f(s,y,u) + f(s,y,v) \right)
      \cdot
      (\nabla\Psi) \Phi
      \\
      &
      + 
      \left( 
        F(s,y,u) - G(t,x,v) + \div  g(t,x,u) -\div f(s,y,v)
      \right)
      \varphi 
      \bigg]
      \\
      &
      \times \mathrm{sign}(u-v)
      \, \mathrm{d}x\,\mathrm{d}t\,\mathrm{d}y\,\mathrm{d}s 
      & \geq & 
      0 \,.
    \end{array}
  \end{equation}
  Introduce a family of functions $\{Y_\vartheta\}_{\vartheta>0}$ as
  in~(\ref{eq:Y}). Let $M = \norma{\pt_u
    g}_{\L\infty(\Omega;\reali^{N})}$ and define $\chi, \psi$ as
  in~(\ref{eq:chipsi}), for $\epsilon, \theta, T_o, R > 0$, $x_o \in
  \reali^N$, (see also Figure~\ref{fig:chipsi}). Remind that with
  these choices, equalities~(\ref{eq:Derivatives}) still hold. Note
  that here the definition of the test function $\phi$ is essentially
  the same as in the preceding proof; the only change is the
  definition of the constant $M$, that is now defined with reference
  to $g$. We also introduce as above the function $\displaystyle
  B(t,x,u,v) = M \modulo{u-v} + \mathrm{sign}(u-v) \left( g(t,x,u) -
    g(t,x,v) \right) \cdot \frac{x-x_o}{\norma{x-x_o}}$ that is
  positive for all $(t,x,u,v) \in \Omega \times \reali^N$, and we
  have:
  \begin{eqnarray*}
    & & 
    \int_{\rpis} \! \int_{\reali^N} \! \int_{\rpis} \! \int_{\reali^N}
    \left[ 
      (u-v) \dt \Phi 
      + 
      \left( g(t,x,u) - g(t,x,v) \right) \cdot \nabla \Phi
    \right]
    \Psi
    \, \mathrm{sign}(u-v) \, \mathrm{d}x \, \mathrm{d}t \, \mathrm{d}y
    \, \mathrm{d}s 
    \\
    & \leq &
    \int_{\rpis} \! \int_{\reali^N} \! \int_{\rpis} \! \int_{\reali^N}
    \left[
      \modulo{u-v} \chi' \psi
      -
      B(t,x,u,v) \chi Y_\theta'
    \right]
    \Psi
    \, \mathrm{d}x \, \mathrm{d}t \, \mathrm{d}y \, \mathrm{d}s
    \\
    & \leq &
    \int_{\rpis} \! \int_{\reali^N} \! \int_{\rpis} \! \int_{\reali^N}
    \modulo{u-v} \, \chi' \, \psi \, \Psi
    \, \mathrm{d}x \, \mathrm{d}t \, \mathrm{d}y \, \mathrm{d}s \,.
  \end{eqnarray*}
  Thanks to the above estimate and~(\ref{eq:sum}), it results
  \begin{displaymath}
    \begin{array}{llcr}
      \displaystyle
      \int_{\rpis} \int_{\reali^N} \int_{\rpis} \int_{\reali^N} 
      &
      \bigg[(u-v) \chi'\psi\Psi 
      & &
      \\
      &
      \displaystyle
      +
      \left( g(t,x,u) - g(t,x,v) - f(s,y,u) + f(s,y,v) \right)
      \cdot
      (\nabla\Psi) \Phi 
      & &
      \\
      &
      \displaystyle 
      + 
      \left( 
        F(s,y,u) - G(t,x,v) + \div  g(t,x,u) - \div f(s,y,v)
      \right)
      \varphi
      \bigg]
      & &
      \\
      &
      \displaystyle
      \times \mathrm{sign}(u-v)
      \,\mathrm{d}x \,\mathrm{d}t\,\mathrm{d}y\,\mathrm{d}s 
      & \geq & 
      0 \,,
    \end{array}
  \end{displaymath}
  i.e.~$I + J_x + J_t + K + L_x + L_t \geq 0$, where
  \begin{eqnarray}
    \label{eq:I}
    I
    & = &
    \int_{\rpis} \int_{\reali^N} \int_{\rpis} \int_{\reali^N} 
    \modulo{u-v} \chi'\psi\Psi 
    \,\mathrm{d}x\,\mathrm{d}t\,\mathrm{d}y\,\mathrm{d}s\, , 
    \\
    \label{eq:Jx}
    J_x
    & = & 
    \int_{\rpis} \int_{\reali^N} \int_{\rpis} \int_{\reali^N}
    \left( f(t,y,v) - f(t,y,u) + f(t,x,u) - f(t,x,v) \right)
    \cdot (\nabla \Psi) \Phi 
    \\
    \nonumber
    & & 
    \qquad \qquad \qquad \qquad
    \times \mathrm{sign}(u-v)
    \, \mathrm{d}x \, \mathrm{d}t \, \mathrm{d}y \, \mathrm{d}s \, ,
    \\
    \nonumber
    J_t
    & = &
    \int_{\rpis} \int_{\reali^N} \int_{\rpis} \int_{\reali^N}
    \left( f(s,y,v) - f(s,y,u) + f(t,y,u) - f(t,y,v) \right)
    \cdot (\nabla \Psi) \Phi
    \\
    \nonumber
    & &
    \qquad \qquad \qquad \qquad
    \times \mathrm{sign}(u-v)
    \,\mathrm{d}x\,\mathrm{d}t\,\mathrm{d}y\,\mathrm{d}s\, ,
    \\
    \label{eq:K}
    K
    & = &
    \int_{\rpis} \int_{\reali^N} \int_{\rpis} \int_{\reali^N}
    \left( (g-f)(t,x,u) - (g-f)(t,x,v) \right)
    \cdot (\nabla\Psi)  \Phi
    \\
    \nonumber
    & &
    \qquad \qquad \qquad \qquad
    \times 
    \mathrm{sign}(u-v)
    \,\mathrm{d}x\,\mathrm{d}t\,\mathrm{d}y\,\mathrm{d}s\,
    ,
    \\
    \label{eq:Lx}
    L_x
    & = &
    \int_{\rpis} \int_{\reali^N} \int_{\rpis} \int_{\reali^N}
    \left( 
      F(t,y,u) - G(t,x,v) + \div g(t,x,u) - \div f(t,y,v)
    \right)
    \varphi 
    \\
    \nonumber
    & &
    \qquad \qquad \qquad \qquad
    \times \mathrm{sign}(u-v)
    \,\mathrm{d}x\,\mathrm{d}t\,\mathrm{d}y\,\mathrm{d}s\, ,
    \\
    \nonumber
    L_t
    & = &
    \int_{\rpis} \int_{\reali^N} \int_{\rpis} \int_{\reali^N}
    \left( 
      F(s,y,u) - F(t,y,u) + \div  f(t,y,v) - \div f(s,y,v)
    \right)
    \varphi 
    \\
    \nonumber
    & &
    \qquad \qquad \qquad \qquad
    \times \mathrm{sign}(u-v) 
    \,\mathrm{d}x\,\mathrm{d}t\,\mathrm{d}y\,\mathrm{d}s
    \,.
  \end{eqnarray}
  Now, we choose $\Psi(t,x) = \nu(t) \, \mu(x)$ as
  in~(\ref{eq:nu}),~(\ref{eq:mu}),~(\ref{eq:Mu}).  Thanks to
  Lemma~\ref{lem:estI}, Lemma~\ref{lem:estJ} and Lemma~\ref{lem:estL}
  we obtain
  \begin{eqnarray}
    \label{eq:estI}
    \limsup_{\epsilon, \pd,\gd \to 0} I 
    & \leq & 
    \int_{\norma{x-x_o} \leq R + MT_o+\theta} \modulo{u(0,x)-v(0,x)} \,\mathrm{d}x
    \\
    \nonumber
    & &
    -
    \int_{\norma{x-x_o}\leq R +M(T_o-T)} \modulo{u(T,x)-v(T,x)} \, \mathrm{d}x\, ,
    \\
    \label{eq:estJ} 
    \limsup_{\epsilon, \pd,\gd \to 0} J_x 
    & \leq &
    N \norma{\nabla\pt_u f }_{\L\infty} 
    \int_0^{T} \int_{B(x_o, R+M(T_o-t)+\theta)} 
    \modulo{v(t,x)-u(t,x)} \, \mathrm{d}x\,\mathrm{d}t\, ,
    \\
    \nonumber
    \limsup_{\epsilon, \pd,\gd\to 0} L_x 
    & \leq &
    \int_0^T \int_{B(x_o,R+M(T_o-t)+\theta)} 
    \norma{\left((F-G)-\div (f-g)\right)(t, y,\cdot)}_{\L{\infty}}
    \, \mathrm{d}y\, \mathrm{d}t
    \\
    \label{eq:estL}
    & &
    +
    \left(
      N \norma{\nabla\pt_u f }_{\L{\infty}} 
      +
      \norma{\pt_u F }_{\L{\infty}} 
      +
      \norma{\pt_u(F-G)}_{\L\infty}
    \right)
    \\
    \nonumber
    & &
    \quad 
    \times \int_0^{T} \int_{B(x_o,  R+M(T_o-t)+\theta)} 
    \modulo{v(t,x)-u(t,x)} \, \mathrm{d}x\,\mathrm{d}t\, .
  \end{eqnarray}
  Besides, we find that:
  \begin{eqnarray*}
    \modulo{J_t}  
    & \leq & 
    \pd \, \norma{\pt_{t} \pt_{u} f}_{\L\infty}
    \int_{\rpis} \int_{\reali^N} \int_{\rpis} \int_{\reali^N}
    \modulo{u(t,x)-u(s,y)}
    \, \norma{\nabla \Psi} \, \Phi
    \, \mathrm{d}x \, \mathrm{d}t \, \mathrm{d}y \, \mathrm{d}s \, ,
    \\
    \modulo{L_t}  
    & \leq &
    \pd \, \omega_N \, (R+M T_o)^N \, (T + \epsilon)
    \left(
      \norma{\pt_t \div f}_{\L\infty}
      +
      \norma{\pt_t F}_{\L\infty}
    \right),
  \end{eqnarray*}
  so that
  \begin{equation}
    \label{eq:uffa}
    \limsup_{\pd\to 0} \modulo{J_t} = \limsup_{\pd\to 0}
    \modulo{L_t} = 0 \,.
  \end{equation}

  In order to estimate $K$ as given in~(\ref{eq:K}), we introduce a
  regularisation of the $y$ dependent functions. In fact, let
  $\rho_\alpha(z) = \frac{1}{\alpha} \rho
  \left(\frac{z}{\alpha}\right)$ and $\sigma_\beta(y) =
  \frac{1}{\beta^N} \sigma \left(\frac{y}{\beta}\right)$, where $\rho
  \in \Cc\infty( \reali; \rpic)$ and $\sigma \in \Cc{\infty}(
  \reali^N; \rpic)$ are such that $\norma{\rho}_{\L1(\reali;\reali)} =
  \norma{\sigma}_{\L1(\reali^N;\reali)} = 1$ and $\mathrm{supp} (\rho)
  \subseteq \left]-1,1\right[$, $\mathrm{supp} (\sigma) \subseteq
  B(0,1)$. Then, introduce
  \begin{displaymath}
    \begin{array}{rcl@{\qquad}rcl}
      P(w)
      & = & 
      (g-f)(t,x,w)\,,
      &
      s_\alpha
      & = &
      \mathrm{sign} \ast_u \rho_\alpha \,,
      \\
      \Upsilon_\alpha^i (w)
      & = &
      s_\alpha(w-v) \, \left( P_i(w) - P_i(v) \right) ,
      &
      u_\beta 
      & = & 
      \sigma_\beta \ast_y u \,,
      \\
      \Upsilon^i (w)
      & = &
      \mathrm{sign} (w-v) \, \left( P_i(w) - P_i(v) \right) ,
    \end{array}
  \end{displaymath}
  so that we obtain
  \begin{eqnarray*}
    & &
    \langle 
    \Upsilon_\alpha^i(u_{\beta}) - \Upsilon_\alpha^i(u), \, \pt_{y_i} \varphi 
    \rangle 
    \\
    & = &
    \int_{\reali^N} \int_{\reali} \mathrm{sign}(w)
    \left( 
      \rho_\alpha(u_\beta-v-w) \, P_i(u_\beta) -  \rho_\alpha(u-v-w) \, P_i(u) 
    \right)
    \pt_{y_i} \varphi 
    \, \mathrm{d}w \, \mathrm{d}y
    \\
    & & 
    -
    \int_{\reali^N} \int_{\reali} \mathrm{sign}(w)
    \left( 
      \rho_\alpha (u_\beta-v-w) - \rho_\alpha (u-v-w)
    \right)
    \, P_i(v) \, \pt_{y_i} \varphi 
    \, \mathrm{d}w \, \mathrm{d}y
    \\
    & = &
    \int_{\reali^N} \int_{\reali} \int_{u}^{u_\beta} 
    \mathrm{sign}(w) \, \rho_\alpha '(U-v-w)
    \left( P_i(U) - P_i(v) \right) \, \pt_{y_i} \varphi
    \, \mathrm{d}U \, \mathrm{d}w \, \mathrm{d}y
    \\
    & &
    +
    \int_{\reali^N} \int_{\reali} \int_{u}^{u_\beta} 
    \mathrm{sign}(w) \,  \rho_\alpha (U-v-w) \, P_i'(U) \, \pt_{y_i} \varphi
    \, \mathrm{d}U \, \mathrm{d}w \, \mathrm{d}y \,.
  \end{eqnarray*}
  Now, we use the relation $\pt_u s_\alpha(u) = \frac{2}{\alpha} \rho
  \left( \frac{u}{\alpha} \right)$ to obtain
  \begin{eqnarray*}
    & &
    \modulo{
      \langle 
      \Upsilon_\alpha^i(u_{\beta}) - \Upsilon_\alpha^i(u), \,
      \pt_{y_i} \varphi 
      \rangle }
    \\
    & \leq & 
    \int_{\reali^N} \frac{2}{\alpha} \sup_{U\in[(u,u_\beta)]}
    \left( 
      \rho \left( \frac{U-v}{\alpha} \right) \left(P_i(U)-P_i(v) \right) 
    \right)
    \min \left\{2\alpha, \modulo{u-u_\beta} \right\} \, \pt_{y_i} \varphi 
    \,\mathrm{d}y
    \\
    & &
    +
    \int_{\reali^N} \int_{u}^{u_\beta} \modulo{P_i'(U)} \pt_{y_i} \varphi 
    \, \mathrm{d}U \, \mathrm{d}y \,.
  \end{eqnarray*}
  When $\alpha$ tends to $0$, thanks to the Dominated Convergence
  Theorem, we obtain
  \begin{eqnarray*}
    \modulo{
      \langle 
      \Upsilon^i(u_{\beta})-\Upsilon^i(u), \pt_{y_i}\varphi
      \rangle 
    }
    & \leq &
    \int_{\reali^N} \modulo{u-u_\beta} \, \norma{P_i'}_{\L\infty}
    \pt_{y_i} \varphi \, \mathrm{d}y.
  \end{eqnarray*}
  Applying the Dominated Convergence Theorem again, we see that
  \begin{eqnarray*}
    \lim_{\beta \to 0} \; \lim_{\alpha \to 0}
    \langle \Upsilon_\alpha^i(u_{\beta}), \, \pt_{y_i} \varphi \rangle
    & = &
    \langle \Upsilon^i(u), \, \pt_{y_i} \varphi \rangle \,,
    \\
    \lim_{\beta \to 0} \; \lim_{\alpha \to 0}
    \langle \Upsilon_\alpha(u_{\beta}), \, \nabla_y\varphi \rangle 
    & = &
    \langle \Upsilon(u),\, \nabla_y\varphi \rangle \,.
  \end{eqnarray*}
  Consequently, it is sufficient to find a bound independent of
  $\alpha$ and $\beta$ on $K_{\alpha,\beta}$, where
  \begin{displaymath}
    K_{\alpha, \beta}
    =
    - \int_{\rpis} \int_{\reali^N} \int_{\rpis} \int_{\reali^N} 
    \Upsilon_\alpha(u_\beta) \cdot \nabla_y \varphi
    \, \mathrm{d}x \, \mathrm{d}t \, \mathrm{d}y \, \mathrm{d}s \,.
  \end{displaymath}
  Integrating by parts, we get
  \begin{eqnarray*}
    K_{\alpha,\beta}
    & = &
    \int_{\rpis} \int_{\reali^N} \int_{\rpis} \int_{\reali^N} \!\!
    \Div_{y} 
    \Upsilon_\alpha (u_\beta) \,\varphi 
    \,\mathrm{d}x \, \mathrm{d}t \, \mathrm{d}y \, \mathrm{d}s
    \\
    & = &
    \int_{\rpis} \int_{\reali^N} \int_{\rpis} \int_{\reali^N}  \!\!
    \pt_u s_\alpha(u_\beta-v)  \nabla u_\beta \cdot
    \left( (g-f) (t,x,u_\beta) - (g-f) (t,x,v) \right) \phi
    \, \mathrm{d}x \,\mathrm{d}t\,\mathrm{d}y\,\mathrm{d}s 
    \\
    & &
    +
    \int_{\rpis} \int_{\reali^N} \int_{\rpis} \int_{\reali^N} 
    s_\alpha (u_\beta-v) 
    \left( \pt_u (g-f) (t,x,u_\beta) \cdot \nabla u_\beta \right) \phi 
    \, \mathrm{d}x \, \mathrm{d}t \, \mathrm{d}y \, \mathrm{d}s
    \\
    & = &
    K_1 + K_2 \,.
  \end{eqnarray*}
  We now search a bound for each term of the sum above.
  \begin{itemize}
  \item For $K_1$, recall that $\pt_u s_\alpha(u) = \frac{2}{\alpha}
    \rho \left( \frac{u}{\alpha} \right)$. Hence, by Dominated
    Convergence Theorem, we get that $K_1\to 0$ when $\alpha\to
    0$. Indeed,
    \begin{eqnarray*}
      & &
      \modulo{\frac{2}{\alpha} \rho
        \left( \frac{u_\beta-v}{\alpha} \right) \, \nabla u_\beta 
        \cdot
        \left( (g-f)(t,x,u_\beta) - (g-f)(t,x,v) \right) \, \varphi} 
      \\
      & \leq &
      \frac{2}{\alpha} \rho\left( \frac{u_\beta-v}{\alpha} \right)
      \varphi
      \norma{\nabla u_\beta(s,y)}
      \int_{v}^{u_\beta} \norma{\pt_u(f-g)(t,x,w)}
      \, \mathrm{d}w
      \\
      & \leq & 
      2\norma{\rho}_{\L\infty(\reali;\reali)} \, \norma{\nabla u_\beta(s,y)} \,
      \norma{\pt_u(f-g)}_{\L\infty(\Omega;\reali^N)} \, \varphi 
      \qquad 
      \in \L1 \left((\rpis\times\reali^N)^2;\reali\right).
    \end{eqnarray*}
  \item Concerning $K_2$,
    \begin{eqnarray*}
      K_2
      & \leq &
      \norma{\pt_u(f-g)}_{\L{\infty}(\Omega; \reali^N)}  
      \int_0^{T+\epsilon+\pd} \int_{\reali^N}
      \norma{\nabla u_\beta(s,y)}
      \, \mathrm{d}y \, \mathrm{d}s
      \\
      & \leq &
      \norma{\pt_u(f-g)}_{\L{\infty}(\Omega; \reali^N)} \,
      \int_0^{T+\varepsilon+\pd} \, \tv(u_\beta(t))\, \mathrm{d}t\,.
    \end{eqnarray*}
  \end{itemize}
  Finally, letting $\alpha,\beta\to 0$ and $\epsilon, \pd, \gd \to 0$,
  thanks to \cite[Proposition~3.7]{AmbrosioFuscoPallara}, we get
  \begin{equation}
    \label{eq:estK}
    \begin{array}{rcl}
      \displaystyle   
      \limsup_{\epsilon, \pd, \gd \to 0}  K 
      & \leq &
      \displaystyle 
      \norma{\pt_u(f-g)}_{\L{\infty}}\int_0^T\tv(u(t))\, \mathrm{d}t \,.
    \end{array}
  \end{equation}

  Now, we collate the estimates obtained in~(\ref{eq:estI}),
  (\ref{eq:estJ}), (\ref{eq:estL}), (\ref{eq:uffa})
  and~(\ref{eq:estK}). Remark the order in which we pass to the
  various limits: first $\epsilon,\pd, \theta\to 0$ and, after, $\gd
  \to 0$. Therefore, we get
  \begin{eqnarray*}
    & &
    \int_{B(x_o,R+M(T_o-T))} \modulo{u(T,x)-v(T,x)} \,
    \mathrm{d}x
    \\
    & \leq & 
    \int_{B(x_o, R +MT_o)} \modulo{u(0,x)-v(0,x)}
    \,\mathrm{d}x
    \\
    & &
    +
    \left[
      2 N \norma{\nabla\pt_u f }_{\L{\infty}} 
      +
      \norma{\pt_u F }_{\L{\infty}}
      +
      \norma{\pt_u(F-G)}_{\L\infty}
    \right]
    \\
    & &
    \qquad\quad 
    \times
    \int_0^{T} \!\! \int_{B(x_o,R+M(T_o-t))} 
    \modulo{v(t,x)-u(t,x)}
    \, \mathrm{d}x \, \mathrm{d}t
    \\
    & &
    +  \biggl[
    \norma{\pt_u(f-g)}_{\L{\infty}} 
    \int_0^T  \tv(u(t))\,  \mathrm{d}t
    \\
    & &
    \qquad
    +
    \int_0^T \int_{B(x_o,R+M(T_o-t))}\!
    \norma{\left((F-G)-\div(f-g)\right)(t, y, \cdot)}_{\L{\infty}}
    \mathrm{d}y\, \mathrm{d}t
    \biggr]
  \end{eqnarray*}
  or equivalently
  \begin{equation}\label{eq:aaa}
    A'(T) \leq A'(0) + \kappa \, A(T) + S(T) \,,
  \end{equation}
  where
  \begin{eqnarray}
    \nonumber
    A(T)
    & = &
    \int_0^{T} \int_{B(x_o, R +M(T_o-t))}
    \modulo{v(t,x)-u(t,x)}
    \, \mathrm{d}x \, \mathrm{d}t\, ,
    \\
    \label{eq:table}
    \kappa
    & = &
    2 N \norma{\nabla\pt_u f }_{\L{\infty}}
    +
    \norma{\pt_u F }_{\L{\infty}}
    +
    \norma{\pt_u (F-G) }_{\L{\infty}}\, ,
    \\
    \nonumber
    S(T)
    & = &
    \norma{\pt_u(f-g)}_{\L{\infty}}
    \int_0^T \tv\left(u(t) \right)\, \mathrm{d}t
    \\
    & &
    \quad
    +
    \int_0^T \int_{B(x_o,R+M(T_o-t))}\!
    \norma{\left((F-G)-\div(f-g)\right)(t,y,\cdot)}_{\L{\infty}}
    \!
    \mathrm{d}y \, \mathrm{d}t .
  \end{eqnarray}
  The bound~(\ref{result}) on $\tv\left(u(t)\right)$ gives:
  \begin{eqnarray*}
    S(T)
    & \leq & 
    \frac{e^{\kappa_o T} -1}{\kappa_o} a
    +
    \int_0^T \frac{e^{\kappa_o (T-t)} -1}{\kappa_o} b(t)\mathrm{d}t 
    +
    \int_0^T c(t)\mathrm{d}t
  \end{eqnarray*}
  where $\kappa_o$ is defined in~(\ref{eq:kappao}) and
  \begin{eqnarray*}
    a
    & = &
    \norma{\pt_u(f-g)}_{\L{\infty}}\tv(u_o) \,,
    \\
    b(t)
    & = & 
    N W_N \norma{\pt_u(f-g)}_{\L{\infty}} 
    \int_{\reali^N}
    \norma{\nabla(F-\div  f)(t,x,\cdot)}_{\L\infty}
    \, \mathrm{d}x \,,
    \\
    c(t)
    & = &
    \int_{B(x_o,R+M(T_o-t))}
    \norma{\left((F-G)-\div(f-g)\right)(t, y,\cdot)}_{\L{\infty}}
    \,\mathrm{d}y  \,,
  \end{eqnarray*}
  since $T \leq T_o$. Consequently
  \begin{equation}
    \label{aab}
    A'(T)
    \leq
    A'(0)
    +
    \kappa A(T)
    + 
    \left(
      \frac{e^{\kappa_o T} - 1}{\kappa_o} a
      +
      \int_0^T \frac{e^{\kappa_o (T-t)} -1}{\kappa_o} b(t)\mathrm{d}t 
      +
      \int_0^T c(t)\mathrm{d}t \right)
    \,.
  \end{equation}
  By a Gronwall type argument, if $\kappa_o = \kappa$, we get
  \begin{displaymath}
    A'(T)
    \leq 
    e^{\kappa T} A'(0)
    +
    T e^{\kappa T} a
    +
    \left(
      \int_0^T  (T-t)e^{\kappa(T-t)} b(t)\, \mathrm{d}t
    \right)
    \left(
      \int_0^T e^{\kappa(T-t)}c(t)\, \mathrm{d}t
    \right) 
  \end{displaymath}
  yielding
  \begin{eqnarray}
    \nonumber
    & &
    \int_{\norma{x-x_o}\leq R} \modulo{u(T,x)-v(T,x)} \, \mathrm{d}x
    \; \leq \;
    e^{\kappa T} \int_{\norma{x-x_o}\leq R +MT}
    \modulo{u_o(x)-v_o(x)} \, \mathrm{d}x
    \\
    \label{eqKKo}
    & + &
    Te^{\kappa T}  \tv(u_o) \, 
    \norma{\pt_u(f-g)}_{\L{\infty}}  
    \\
    \nonumber
    & + &
    NW_N 
    \left(
      \int_0^T (T-t)e^{\kappa(T-t)}
      \int_{\reali^N}
      \norma{\nabla(F-\div f)(t, x,\cdot)}_{\L\infty}
      \mathrm{d}x \, \mathrm{d}t 
    \right)
    \norma{\pt_u(f-g)}_{\L{\infty}}
    \\
    \nonumber
    & + &
    \int_0^T e^{\kappa (T-t)} \int_{\norma{x-x_o}\leq R+M(T-t)}
    \norma{\left((F-G) - \div(f-g) \right)(t,x,\cdot)}_{\L{\infty}} 
    \, \mathrm{d}x\, \mathrm{d}t
  \end{eqnarray}
  while, in the case $\kappa_o \neq \kappa$, we have
  \begin{eqnarray*}
    A'(T)
    & \leq & 
    e^{\kappa T} A'(0)
    +
    \frac{e^{\kappa_o T}-e^{\kappa T}}{\kappa_o-\kappa} \, a
    +
    \int_0^T \frac{e^{\kappa_o (T-t)}-e^{\kappa (T-t)}}{\kappa_o-\kappa}\, b(t)
    \, \mathrm{d}t 
    +
    \int_0^T e^{\kappa(T-t)}c(t)\, \mathrm{d}t\, .
  \end{eqnarray*}
  Taking $T=T_o$, we finally obtain the result.
\end{proofof}

\begin{remark}
  Assuming that also $(g,G)$ satisfies~\textbf{(H2)}, allows us to
  exchange the role of $u$ and $v$ in~(\ref{eq:table}). Let
  \begin{eqnarray*}
    \tilde \kappa_o
    & = &
    N W_N
    \left( 
      (2N+1) 
      \norma{\nabla\pt_u g}_{\L{\infty}} 
      + 
      \norma{\pt_u G}_{\L{\infty}}
    \right) \,,
    \\
    \tilde a
    & = &
    \norma{\pt_u(f-g)}_{\L{\infty}} \tv(v_o) \,,
    \\
    \tilde b(t)
    & = &
    \norma{\pt_u(f-g)}_{\L{\infty}}
    N W_ N  \int_{\reali^N}
    \norma{\nabla(G-\div g)(t,x,\cdot)}_{\L\infty} \, \mathrm{d}x  \,,
    \\
    \tilde \kappa
    & = &
    2 N 
    \norma{\nabla\pt_u g}_{\L{\infty}}+\norma{\pt_u G}_{\L{\infty}} 
    +
    \norma{\pt_u (F-G)}_{\L{\infty}}  \,,
  \end{eqnarray*}
  and repeating the same computations as above, we obtain
  \begin{displaymath}
    A'(T)
    \leq 
    A(0) + \tilde \kappa A(T)+
    \left(
      \frac{e^{\tilde\kappa_o T} - 1}{\tilde\kappa_o}\, \tilde a
      +
      \int_0^T 
      \frac{e^{\tilde\kappa_o (T-t)} -1}{\tilde\kappa_o} \, \tilde
      b(t)
      \, \mathrm{d}t 
      +
      \int_0^T c(t)\, \mathrm{d}t \right)
  \end{displaymath}
  so that, finally,
  \begin{eqnarray*}
    A'(T)
    & \leq &
    A'(0)
    +
    \min(\kappa,\tilde\kappa) A(T) 
    + 
    \max 
    \left[ 
      \frac{e^{\kappa_o T} - 1}{\kappa_o}\, a
      +
      \int_0^T \frac{e^{\kappa_o (T-t)} -1}{\kappa_o} \,  b(t)\,
      \mathrm{d}t ,
    \right.
    \\
    & &
    \left.
      \frac{e^{\tilde\kappa_o T} - 1}{\tilde\kappa_o}\, \tilde a
      +
      \int_0^T 
      \frac{e^{\tilde\kappa_o (T-t)} -1}{\tilde\kappa_o} \, \tilde b(t)
      \, \mathrm{d}t 
    \right]
    + \int_0^T c(t)\mathrm{d}t.
  \end{eqnarray*}
\end{remark}

We collect below some lemmas that were used in the previous proof. The
first one reminds a part of the proof
of~\cite[Theorem~2.1]{bouchutperthame}.

\begin{lemma}
  \label{lem:estI}
  Let $I$ be defined as in~(\ref{eq:I}). Then,
  \begin{eqnarray*}
    \limsup_{\epsilon\to0} I 
    & \leq &
    \int_{\norma{x-x_o} \le R+MT_o+\theta}
    \modulo{u(0,x)-v(0,x)} \, \mathrm{d}x
    \\
    & &
    -
    \int_{\norma{x-x_o}\le R +M(T_o-T)} 
    \modulo{u(T,x)-v(T,x)} \, \mathrm{d}x 
    + 
    2 \sup_{\tau\in\{0,T\}} \tv \left(u(\tau) \right) \gd
    \\
    & &      
    +
    2 \sup_{t\in \{0,T\} \atop s\in\left]t,t+\pd\right[}
    \int_{\norma{y-x_o}\le R +\gd+M(T_o-t)+\theta}
    \modulo{u(t,y)-u(s,y)} \, \mathrm{d}y \,.
  \end{eqnarray*}
\end{lemma}

\begin{proof}
  By the triangle inequality $I \leq I_1+I_2+I_3$, with
  \begin{eqnarray*}
    I_1
    & = &
    \int_{\rpis} \int_{\reali^N}\int_{\rpis} \int_{\reali^N}
    \modulo{u(t,x)-v(t,x)} 
    \, \chi'(t) \, \psi(t,x) \, \Psi(t-s,x-y)
    \, \mathrm{d}x \, \mathrm{d}t \, \mathrm{d}y \, \mathrm{d}s \,,
    \\
    I_2
    & = &
    \int_{\rpis} \int_{\reali^N}\int_{\rpis} \int_{\reali^N} 
    \modulo{u(t,x)-u(t,y)} 
    \, \modulo{\chi'(t)} \, \psi(t,x) \, \Psi(t-s,x-y)
    \,\mathrm{d}x \, \mathrm{d}t \, \mathrm{d}y \, \mathrm{d}s  \,,
    \\
    I_3
    & = &
    \int_{\rpis} \int_{\reali^N}\int_{\rpis} \int_{\reali^N} 
    \modulo{u(t,y)-u(s,y)}
    \, \modulo{\chi'(t)} \, \psi(t,x) \, \Psi(t-s,x-y)
    \, \mathrm{d}x \, \mathrm{d}t \, \mathrm{d}y \, \mathrm{d}s\, .
  \end{eqnarray*}
  Then,
  \begin{eqnarray*}
    I_1 
    & = &
    \int_{\rpis} \int_{\reali^N}
    \modulo{u(t,x)-v(t,x)}
    \left(Y_\epsilon '(t) - Y_\epsilon '(t-T) \right) 
    \psi(t,x) 
    \, \mathrm{d}x \, \mathrm{d}t
    \\
    & \leq &
    \int_{\rpis} \int_{\norma{x-x_o}\le R +M(T_o-t)+\theta}
    \modulo{u(t,x)-v(t,x)} \, Y_\epsilon'(t)
    \, \mathrm{d}x \, \mathrm{d}t
    \\
    & &
    -
    \int_{\rpis}\int_{\norma{x-x_o} \le R +M(T_o-t)}
    \modulo{u(t,x)-v(t,x)} \, Y_\epsilon'(t-T)
    \, \mathrm{d}x \, \mathrm{d}t
  \end{eqnarray*}
  and by the $\L1$ right continuity of $u$ and $v$ in time, thanks to
  Theorem~\ref{teo:kruzkov}
  \begin{eqnarray*}
    \limsup_{\epsilon\to 0} I_1 
    & \leq &  
    \int_{\norma{x-x_o}\le R +MT_o+\theta}
    \modulo{u(0,x)-v(0,x)} 
    \, \mathrm{d}x
    \\
    & &
    -
    \int_{\norma{x-x_o}\le R +M(T_o-T)}
    \modulo{u(T,x)-v(T,x)}
    \, \mathrm{d}x \,.
  \end{eqnarray*}
  For $I_2$ and $I_3$, we have
  \begin{eqnarray*}
    I_2 
    & \leq &
    \int_{\rpis} \! \int_{\reali^N}\!
    \int_{\norma{x-x_o}\le R+M(T_o-t)+\theta}
    \modulo{u(t,x)-u(t,y)} (Y_\epsilon'(t) +Y_\epsilon'(t-T))
    \mu
    \, \mathrm{d}x \, \mathrm{d}y \, \mathrm{d}t \,,
    \\
    I_3 
    & \leq & 
    \int_{\rpis} \! \int_{\rpis} \! 
    \int_{\norma{y-x_o}\le R+\gd+M(T_o-t)+\theta}
    \modulo{u(t,y)-u(s,y)} 
    \left( Y_\epsilon'(t) + Y_\epsilon'(t-T) \right)
    \nu
    \,\mathrm{d}y \,\mathrm{d}s\,\mathrm{d}t\, .
  \end{eqnarray*}
  As $\epsilon \to 0$, we use on the one hand the $\L1$ right
  continuity in time of $u$, thanks to Theorem~\ref{teo:kruzkov}, and
  on the other hand that $u(t) \in \BV(\reali^N; \reali)$, thanks to
  Theorem~\ref{teo:tv}. In particular, we can use~(\ref{eq:afp}) to
  obtain
  \begin{eqnarray*}
    \limsup_{\epsilon\to 0} I_2
    & \leq & 
    \sum_{t=0,T} \sup_{\norma{h}\leq \gd}
    \int_{\norma{x-x_o}\le R +M(T_o-t)+\theta}
    \modulo{u(t,x)-u(t,x+h)}
    \, \mathrm{d}x
    \\
    & \leq & 
    2 \sup_{\norma{h}\le \gd \atop t\in\{0,T\}}
    \, \int_{\norma{x-x_o}\le R +M(T_o-t)+\theta}
    \modulo{u(t,x)-u(t,x+h)} 
    \, \mathrm{d}x
    \\
    & \leq & 
    2 \sup_{t\in\{0,T\}} \tv\left(u(t)\right) \gd\, ,
    \\
    \limsup_{\epsilon\to 0} I_3
    & \leq &
    \sum_{t=0,T} \sup_{s \in \left]t,t+\pd\right[}
    \int_{\norma{y-x_o}\le R+\gd+M(T_o-t)+\theta}
    \modulo{u(t,y)-u(s,y)}
    \, \mathrm{d}y
    \\
    & \leq & 
    2
    \sup_{t\in \{0,T\} \atop s\in\left]t,t+\pd\right[}
    \int_{\norma{y-x_o}\le R +\gd+M(T_o-t)+\theta}
    \modulo{u(t,y)-u(s,y)}
    \, \mathrm{d}y \,.
  \end{eqnarray*}
\end{proof}

\begin{lemma}
  \label{lem:estJ}
  Let $J_x$ be defined as in~(\ref{eq:Jx}). Then,
  \begin{eqnarray*}
    \limsup_{\epsilon\to 0} J_x 
    & \leq &
    N
    \norma{\nabla\pt_u f }_{\L\infty} 
    \int_0^{T}\int_{B(x_o,R+M(T_o-t)+\theta)} 
    \modulo{v(t,x)-u(t,x)}\mathrm{d}x
    \, \mathrm{d}t
    \\
    & &
    +
    N T \norma{\nabla\pt_u f }_{\L\infty} 
    \sup_{\tau\in[0,T]} \tv \left(u(\tau) \right) \gd
    \\
    & &
    +
    N T \norma{\nabla\pt_u f }_{\L\infty} 
    \sup_{t\in [0,T] \atop s\in\left]t,t+\pd\right[}
    \int_{\norma{y-x_o}\le R +\gd+M(T_o-t)+\theta}
    \modulo{u(t,y)-u(s,y)}
    \, \mathrm{d}y
    \,.
  \end{eqnarray*}
\end{lemma}

\begin{proof}
  By assumptions~\textbf{(H1)}, $f \in \C2(\Omega; \reali^N)$ and
  therefore
  \begin{eqnarray*}
    & &
    \norma{f(t,y,v)-f(t,y,u)+f(t,x,u)-f(t,x,v)}
    \\
    & = &
    \norma{
      \int_{u(s,y)}^{v(t,x)} \int_0^1 \nabla \pt_u f
      \left(t,x(1-r)+ry, w\right)
      \cdot (y-x) 
      \, \mathrm{d}r \, \mathrm{d}w
    }
    \\
    & \leq &
    \norma{\nabla\pt_u f}_{\L\infty} \, \norma{x-y} 
    \, \modulo{v(t,x)-u(s,y)} \,.
  \end{eqnarray*}
  Then,
  \begin{eqnarray*}
    J_x 
    & \leq & 
    % \norma{\nabla\pt_u f}_{\L\infty} \int_{\rpis} \int_{\reali^N}
    % \int_{\rpis} \int_{\reali^N} \!\!  \modulo{v(t,x)-u(s,y)} \,
    % \norma{x-y} \, \norma{\nabla\Psi} \, \chi \, \psi \,\mathrm{d}x
    % \,\mathrm{d}t\,\mathrm{d}y\,\mathrm{d}s
    % \\
    % & \leq &
    \norma{\nabla\pt_{u} f}_{\L\infty} \int_{\rpis}
    \int_{\reali^N} \int_{\rpis} \int_{\reali^N}
    \modulo{v(t,x)-u(s,y)} \, \norma{x-y} \, \norma{\nabla\mu} \, \nu
    \, \chi \, \psi
    \,\mathrm{d}x \,\mathrm{d}t\,\mathrm{d}y\,\mathrm{d}s \,.
  \end{eqnarray*}
  Similarly to the proof of Lemma~\ref{lem:estI}, we apply the
  triangle inequality and obtain $J_x \leq J_1 + J_2 + J_3$ where
  \begin{eqnarray*}
    J_1
    & = &
    \norma{\nabla\pt_u f}_{\L\infty}
    \int_{\rpis} \int_{\reali^N} \int_{\rpis} \int_{\reali^N} 
    \modulo{v(t,x)-u(t,x)} \, \norma{x-y} \, \norma{\nabla\mu}
    \, \nu \, \chi \, \psi
    \, \mathrm{d}x \,\mathrm{d}t \,\mathrm{d}y \, \mathrm{d}s \,,
    \\
    J_2
    & = &
    \norma{\nabla\pt_u f }_{\L\infty}
    \int_{\rpis} \int_{\reali^N} \int_{\rpis} \int_{\reali^N} 
    \modulo{u(t,x)-u(t,y)} \, \norma{x-y} \, \norma{\nabla\mu}
    \, \nu \, \chi \, \psi
    \, \mathrm{d}x \,\mathrm{d}t \, \mathrm{d}y \, \mathrm{d}s \,,
    \\
    J_3
    & = & 
    \norma{\nabla\pt_u f }_{\L\infty}
    \int_{\rpis} \int_{\reali^N} \int_{\rpis} \int_{\reali^N}  
    \modulo{u(t,y)-u(s,y)} \, \norma{x-y} \, \norma{\nabla\mu}
    \, \nu \, \chi \, \psi
    \, \mathrm{d}x \,\mathrm{d}t \, \mathrm{d}y \, \mathrm{d}s\,.
  \end{eqnarray*}
  For $J_1$, we have, thanks to~(\ref{eq:mu3})
  \begin{displaymath}
    J_1
    \leq 
    N \norma{\nabla\pt_u f }_{\L\infty}
    \int_0^{T+\epsilon}\int_{B(x_o, R+M(T_o-t)+\theta)}
    \modulo{v(t,x)-u(t,x)}
    \, \mathrm{d}x \, \mathrm{d}t\, .
  \end{displaymath}
  For $J_2$, we have
  \begin{eqnarray*}
    J_2 
    & \leq & 
    N \norma{\nabla\pt_u f }_{\L\infty}
    \int_0^{T+\epsilon}\sup_{\norma{h}\le \gd} 
    \int_{\norma{x-x_o}\le R +M(T_o-t)+\theta}
    \modulo{u(t,x)-u(t,x+h)}
    \, \mathrm{d}x \, \mathrm{d}t
    \\
    & \leq & 
    N \norma{\nabla\pt_u f }_{\L\infty} (T+\epsilon)
    \sup_{\tau\in[0,T+\epsilon]} \tv \left(u(\tau) \right) \, \gd
    \, ,
  \end{eqnarray*}
  and for $J_3$
  \begin{eqnarray*}
    J_3 
    & \leq & 
    N \norma{\nabla\pt_u f }_{\L\infty}
    \int_0^{T+\epsilon} \sup_{s \in \left]t, t+\pd\right[}
    \int_{\norma{x-x_o}\le R +\gd+M(T_o-t)+\theta}
    \modulo{u(t,y)-u(s,y)}
    \, \mathrm{d}y \, \mathrm{d}t
    \\
    & \leq &
    N \norma{\nabla\pt_u f }_{\L\infty} 
    (T+\epsilon) 
    \sup_{t\in [0,T+\epsilon] \atop s\in\left]t,t+\pd\right[}
    \int_{\norma{y-x_o}\le R +\gd+M(T_o-t)+\theta}
    \modulo{u(t,y)-u(s,y)}
    \, \mathrm{d}y \,.
  \end{eqnarray*}
  In particular, letting $\gd, \pd,\epsilon, \theta \to 0$, we prove
  that $J_2, J_3\to 0$ and
  \begin{displaymath}
    \limsup_{\gd, \pd,\epsilon, \theta \to 0} J_1
    \leq
    N \norma{\nabla\pt_u f}_{\L\infty}
    \int_0^{T} \int_{B(x_o, R+M(T_o-t))}
    \modulo{v(t,x)-u(t,x)}
    \, \mathrm{d}x \, \mathrm{d}t
  \end{displaymath}
  completing the proof.
\end{proof}

\begin{lemma}
  \label{lem:estL}
  Let $L_x$ be defined as in~(\ref{eq:Lx}) and $M_1$ as
  in~(\ref{eq:M1}). Then
  \begin{eqnarray*}
    \limsup_{\epsilon\to 0} L_x 
    & \leq &
    T \int_0^T \int_{\norma{x-x_o}\leq R+M(T_o-t)+\theta} 
    \norma{\left((F-G)-\div(f-g)\right)(t,x,\cdot)}_{\L{\infty}}
    \, \mathrm{d}x\, \mathrm{d}t
    \\
    & &
    +
    \left(
      N \norma{\nabla\pt_u f }_{\L{\infty}}
      +
      \norma{\pt_u F }_{\L{\infty}}
      +
      \norma{\pt_u(F-G)}_{\L\infty} 
    \right)
    \\
    & &
    \quad 
    \times
    \left[
      \int_0^{T} \int_{B(x_o, R+M(T_o-t)+\theta)}
      \modulo{v(t,x)-u(t,x)}
      \, \mathrm{d}x \, \mathrm{d}t\right.
    \\
    & &
    \qquad 
    + 
    T \sup_{\tau\in[0,T]}\tv(u(\tau))\gd 
    \\
    & &
    \qquad 
    \left.
      +
      T 
      \sup_{t\in [0,T] \atop s\in\left]t,t+\pd\right[}
      \int_{\norma{y-x_o}\le R +\gd+M(T_o-t)+\theta}
      \modulo{u(t,y)-u(s,y)}
      \, \mathrm{d}y
    \right]
    \\
    & &
    +
    \gd \, M_1
    \int_0^T \int_{\reali^N} 
    \norma{\nabla (F-\div f)(t,x,\cdot)}_{\L{\infty}}
    \mathrm{d}x \, \mathrm{d}t
    \,.
  \end{eqnarray*}
\end{lemma}

\begin{proof}
  Let
  \begin{eqnarray*}
    L_1
    & = &
    \int_{\rpis} \int_{\reali^N} \int_{\rpis} \int_{\reali^N}  
    \left((F - G) - \div  (f-g)\right)(t,x,u) 
    \, \varphi 
    \, \mathrm{sign}(u-v) \, \mathrm{d}x \, \mathrm{d}t \, \mathrm{d}y
    \, \mathrm{d}s \,,
    \\
    L_2
    & = &
    \int_{\rpis} \int_{\reali^N} \int_{\rpis} \int_{\reali^N}  
    \left((F-G)(t,x,v)-(F-G)(t,x,u) \right)
    \, \varphi 
    \, \mathrm{sign}(u-v) \, \mathrm{d}x \, \mathrm{d}t \, \mathrm{d}y
    \,\mathrm{d}s \,,
    \\
    L_3
    & = &
    \int_{\rpis} \int_{\reali^N} \int_{\rpis} \int_{\reali^N}
    \left(
      F(t,y,u) - F(t,y,v) + \div  f(t,x,u) - \div  f(t,x,v)
    \right)
    \varphi
    \\
    & &
    \qquad\qquad \qquad
    \times 
    \,\mathrm{sign}(u-v) \, \mathrm{d}x \, \mathrm{d}t \, \mathrm{d}y
    \, \mathrm{d}s \,,
    \\
    L_4
    & = &
    \int_{\rpis} \int_{\reali^N} \int_{\rpis} \int_{\reali^N}  
    \left( (F-\div f)(t,y,v)-(F-\div  f)(t,x,v) \right) 
    \, \varphi\, \mathrm{sign}(u-v)
    \, \mathrm{d}x \,\mathrm{d}t \,\mathrm{d}y \,\mathrm{d}s,
  \end{eqnarray*}
  so that $L_x = L_1 + L_2 + L_3 + L_4$. Clearly,
  \begin{displaymath}
    L_1 
    \leq
    \int_0^{T+\varepsilon} \int_{\norma{x-x_o}\leq R+M(T_o-t)+\theta}
    \norma{\left((G-F)-\div(f-g)\right)(t, x,\cdot)}_{\L{\infty}}
    \mathrm{d}x \,\mathrm{d}t\, .
  \end{displaymath}
  For $L_2$ and $L_3$, we have
  \begin{eqnarray*}
    L_2 
    &\leq & 
    \norma{\pt_u(F-G)}_{\L\infty} \int_{\rpis} \int_{\reali^N}
    \int_{\rpis} \int_{\reali^N} 
    \modulo{u(s,y) - v(t,x)} \,\varphi 
    \, \mathrm{d}x \, \mathrm{d}t \, \mathrm{d}y \, \mathrm{d}s \,,
    \\
    L_3
    & = &
    \int_{\rpis} \int_{\reali^N} \int_{\rpis} \int_{\reali^N} 
    \mathrm{sign}(u-v) 
    \left(
      \int_{v}^{u} 
      \left(\pt_u \div f(t,x,w) + \pt_u F(t,y,w) \right) 
      \mathrm{d}w
    \right)
    \varphi
    \,\mathrm{d}x \,\mathrm{d}t\,\mathrm{d}y\,\mathrm{d}s
    \\
    &\leq &
    \left(
      N \norma{\nabla\pt_u f}_{\L{\infty}}
      +
      \norma{\pt_u F}_{\L\infty}
    \right)
    \int_{\rpis} \int_{\reali^N} \int_{\rpis} \int_{\reali^N}
    \modulo{v(t,x)-u(s,y)} \varphi
    \, \mathrm{d}x \, \mathrm{d}t \, \mathrm{d}y \, \mathrm{d}s \,.
  \end{eqnarray*}
  Proceeding as for $J_x$, we find the following bound for $\iiiint
  \modulo{v(t,x)-u(s,y)}\varphi$ in $L_2$, $L_3$.
  \begin{eqnarray*}
    L_2 + L_3
    & \!\leq \!&
    \left( 
      N \norma{\nabla\pt_u f }_{\L{\infty}}
      +
      \norma{\pt_u F}_{\L\infty}
      +
      \norma{\pt_u(F-G)}_{\L\infty}
    \right)
    \\
    & & 
    \times 
    \biggl[
    \int_0^{T+\epsilon} \!\!\!\! \int_{B(x_o,R+M(T_o-t)+\theta)} \!\!
    \modulo{v(t,x)-u(t,x)} 
    \mathrm{d}x \mathrm{d}t
    +(T+\epsilon) \sup_{\tau\in[0,T+\epsilon]} \!\!\!\!
    \tv\left(u(\tau) \right) \gd
    \\
    &  &
    \qquad \quad
    +
    (T+\epsilon) 
    \sup_{t\in [0,T+\epsilon] \atop s\in\left]t,t+\pd\right[}
    \int_{\norma{y-x_o}\le R+\gd+M(T_o-t)+\theta}
    \modulo{u(t,y)-u(s,y)}
    \, \mathrm{d}y
    \biggr]\,.
  \end{eqnarray*}
  For $L_4$ we have
  \begin{eqnarray*}
    L_4
    & = &
    \int_{\rpis} \int_{\reali^N} \int_{\rpis} \int_{\reali^N} 
    \left[
      \int_0^1 \nabla (F-\div f) \left(t,rx+(1-r)y, v\right) \cdot
      (y-x) 
      \, \mathrm{d}r 
    \right] \varphi
    \\
    & &
    \qquad \qquad \qquad
    \times 
    \mathrm{sign}(u-v)
    \, \mathrm{d}x \, \mathrm{d}t \, \mathrm{d}y \, \mathrm{d}s
    \\
    & \leq &
    \gd \, M_1
    \int_0^{T+\epsilon} \int_{\reali^N} 
    \norma{\nabla (F-\div f)(t,x,\cdot)}_{\L{\infty}}
    \mathrm{d}x \, \mathrm{d}t
    \,.
  \end{eqnarray*}
  To complete the proof, it is sufficient to note that $L_x = L_1 +
  L_2 + L_3 + L_4$.
\end{proof}

\small{

  \bibliography{crm}

  \bibliographystyle{abbrv} }

\end{document}